\documentclass[11pt,a4paper]{amsart}

 \usepackage{amsfonts,amsmath,amscd,amssymb,amsbsy,amsthm,amstext,amsopn}
 \usepackage{fullpage,mathrsfs,subfigure}
 \usepackage{pstricks,pst-node,pst-text,pst-tree}
 \usepackage{stmaryrd}  
 \usepackage{extarrows}
\usepackage[all]{xy}

 \numberwithin{equation}{section}
\newtheorem{theorem}{Theorem}[section]
\newtheorem{proposition}[theorem]{Proposition}
\newtheorem{lemma}[theorem]{Lemma}

\newtheorem{corollary}[theorem]{Corollary}

\theoremstyle{definition}
\newtheorem{definition}[theorem]{Definition}
\newtheorem{example}[theorem]{Example}

\theoremstyle{remark}
\newtheorem{remark}[theorem]{Remark}

\newcommand{\kk}{\ensuremath{\Bbbk}}

\newcommand{\NN}{\ensuremath{\mathbb{N}}} 
 
\newcommand{\QQ}{\ensuremath{\mathbb{Q}}} 
 
\newcommand{\ZZ}{\ensuremath{\mathbb{Z}}} 
\newcommand{\one}{\ensuremath{(\mathrm{i})}}
\newcommand{\two}{\ensuremath{(\mathrm{ii})}}
\newcommand{\three}{\ensuremath{(\mathrm{iii})}}

\newcommand{\coh}{\operatorname{coh}}

\renewcommand{\div}{\operatorname{div}} 
\newcommand{\head}{\operatorname{\mathsf{h}}}

\newcommand{\git}{\ensuremath{/\!\!/\!}}

\newcommand{\inc}{\operatorname{inc}}

\newcommand{\pic}{\operatorname{pic}}
\newcommand{\rad}{\operatorname{rad}}

\newcommand{\supp}{\operatorname{supp}}
\newcommand{\tail}{\operatorname{\mathsf{t}}}

 \newcommand{\Cl}{\operatorname{Cl}}  

\newcommand{\Cox}{\operatorname{Cox}}
\newcommand{\End}{\operatorname{End}}
\newcommand{\Ext}{\operatorname{Ext}} 
  
\newcommand{\Gr}{\operatorname{Gr}}  
\newcommand{\Hom}{\operatorname{Hom}}

\newcommand{\Ker}{\operatorname{ker}}

\newcommand{\Pic}{\operatorname{Pic}}
\newcommand{\Proj}{\operatorname{Proj}}

\newcommand{\Spec}{\operatorname{Spec}} 
 
\newcommand{\Wt}{\operatorname{Wt}} 

 \DeclareMathOperator{\Rderived}{\mathbf{R}\!}
 
 \newcommand{\modAL}{\operatorname{mod-}\ensuremath{\!A}}
\newcommand{\modA}{\operatorname{mod-}\ensuremath{\!A}}

 \newcommand{\modALk}{\operatorname{mod-}\ensuremath{\!A_k}}

\title{Mori Dream Spaces as fine moduli of quiver representations}
\thanks{MSC 2010: Primary  14D22, 16G20; Secondary 14F05, 16E35, 16S38} 

 \author{Alastair Craw and Dorothy Winn} 
 \address{Department of Mathematics\\ University of Glasgow\\ Glasgow\\ G12 8QW\\ United Kingdom}
 \email{Alastair.Craw@glasgow.ac.uk, Dorothy.Winn@glasgow.ac.uk}

\begin{document}
\bibliographystyle{plain}

 \begin{abstract}
 We construct Mori Dream Spaces as fine moduli spaces of representations of bound quivers, thereby extending results of Craw--Smith~\cite{CrawSmith} beyond the toric case. Any collection of effective line bundles $\mathscr{L}=(\mathscr{O}_X, L_1,\dots, L_r)$ on a Mori Dream Space $X$ defines a bound quiver of sections and a map from $X$ to a toric quiver variety $\vert\mathscr{L}\vert$ called the multigraded linear series. We provide necessary and sufficient conditions for this map to be a closed immersion and, under additional assumptions on $\mathscr{L}$, the image realises $X$ as the fine moduli space of $\vartheta$-stable representations of the bound quiver. As an application, we show how to reconstruct del Pezzo surfaces from a full, strongly exceptional collection of line bundles.
    \end{abstract}
    
 \maketitle
 
 \tableofcontents

 \section{Introduction}
Mori Dream Spaces and their Cox rings have been the subject of a great deal of interest since their introduction by Hu--Keel~\cite{HuKeel} over a decade ago. From the geometric side, these varieties enjoy the property that all operations of the Mori programme can be carried out by variation of GIT quotient, while from the algebraic side, obtaining an explicit presentation of the Cox ring is an interesting problem in itself. Examples of Mori Dream Spaces include $\QQ$-factorial projective toric varieties, spherical varieties and log Fano varieties of arbitrary dimension. In this paper we use the representation theory of quivers to study multigraded linear series on Mori Dream Spaces. Our main results construct Mori Dream Spaces as fine moduli spaces of $\vartheta$-stable representations of bound quivers for a special stability parameter $\vartheta$, thereby extending results of Craw--Smith~\cite{CrawSmith} for projective toric varieties. 

Let $X$ be a projective variety and $\mathscr{L}=(\mathscr{O}_X, L_1, \dots, L_r)$ a collection of distinct, effective line bundles on $X$. The associated \emph{multigraded linear series} $\vert\mathscr{L}\vert$  is a smooth projective toric variety that provides a multigraded analogue of the classical linear series of a single line bundle. To construct $\vert \mathscr{L}\vert$ one first defines the bound quiver of sections of $\mathscr{L}$ to be a finite, acyclic quiver $Q$ together with an ideal of relations $J_\mathcal{L}$ in the path algebra $\kk Q$ for which $\kk Q/J_\mathcal{L}$ is isomorphic to $\End(\bigoplus_{0\leq i\leq r}L_i)$. Setting aside the ideal of relations for now, the multigraded linear series $\vert \mathscr{L}\vert$ is defined to be the toric quiver variety obtained as the fine moduli space of $\vartheta$-stable representations of $Q$ with dimension vector $(1,\dots, 1)$ for the special stability parameter $\vartheta=(-r, 1, \dots, 1)$.  Paths in the quiver arise from spaces of sections of the form $H^0(X, L_j\otimes L_i^{-1})$ for $0\leq i, j\leq r$, and evaluating these sections defines a morphism
\begin{equation}
\label{eqn:morphismintro}
\varphi_{\vert\mathscr{L}\vert}\colon X\longrightarrow \vert\mathscr{L}\vert
\end{equation}
 when each $L_i$ is basepoint-free. In the case where $X$ is a projective toric variety, Craw--Smith~\cite{CrawSmith} constructed the quiver of sections using only torus-invariant sections of the bundles  $L_j\otimes L_i^{-1}$, and described precisely the image of $\varphi_{\vert \mathscr{L}\vert}$.  The bound quiver of sections and a morphism $\varphi_{\vert\mathscr{L}\vert}$ are constructed in general by Craw~\cite[Theorem~5.4]{Craw}, but one cannot study the image without choosing bases for the appropriate spaces of sections.
  
 Our first main result solves this problem for Mori Dream Spaces by choosing a suitable spanning set for each space of sections. Every Mori Dream Space $X$ admits a closed immersion into a $\QQ$-factorial projective toric variety $\widetilde{X}$, such that each line bundle on $X$ is the restriction of a reflexive sheaf of rank one on $\widetilde{X}$. In particular, the collection $\mathscr{L}$ on $X$ can be lifted uniquely to a collection $\widetilde{\mathscr{L}}=(\mathscr{O}_{\widetilde{X}},E_1,\dots, E_r)$ of rank one reflexive sheaves on $\widetilde{X}$. While the spaces $H^0(X, L_j\otimes L_i^{-1})$ have no preferred basis, the restriction to $X$ of the torus-invariant sections of the sheaf $(E_j\otimes E_i^{\vee})^{\vee\vee}$ provide a spanning set for $H^0(X, L_j\otimes L_i^{-1})$ that we use to build the paths in the quiver $Q$ from vertex $i$ to vertex $j$. That is, we \emph{define} the quiver $Q$ using differences of sections of bundles from the collection $\widetilde{\mathscr{L}}$, and we introduce an ideal of relations $J_\mathscr{L}$ in $\kk Q$ for which $\kk Q/J_\mathscr{L}\cong \End(\bigoplus_{0\leq i\leq r}L_i)$. Our first main result describes the rational map 
 \[
 \varphi_{\vert\mathscr{L}\vert}\colon X\dashrightarrow \vert\mathscr{L}\vert
 \]
obtained by evaluating sections, and establishes the link between the collection $(\mathscr{W}_0, \dots, \mathscr{W}_r)$ of tautological lines bundles on the fine moduli space $\vert \mathscr{L}\vert$ and the collection $\mathscr{L}$ of line bundles on $X$ (see Section~\ref{sec:theorem1} for the proof):

\begin{theorem}
\label{thm:1}
For a collection $\mathscr{L}=(\mathscr{O}_X,L_1,\dots, L_r)$ of distinct, effective line bundles on $X$, the map $\varphi_{\vert\mathscr{L}\vert}\colon X\dashrightarrow \vert\mathscr{L}\vert$ is a morphism if and only of each $L_i$ is basepoint-free. Moreover, the image is presented explicitly as a geometric quotient, and the tautological bundles on $\vert \mathscr{L}\vert$ satisfy $\varphi_{\vert\mathscr{L}\vert}^*(\mathscr{W}_i)=L_i$.
\end{theorem}

\noindent In addition, we provide a necessary and sufficient condition (that is easy to implement) for the morphism $\varphi_{\vert\mathscr{L}\vert}\colon X\rightarrow \vert\mathscr{L}\vert$ to be a closed immersion, thereby improving upon a sufficient condition in the toric case by Craw-Smith~\cite[Corollary~4.10]{CrawSmith}. A straightforward application of multigraded regularity introduced by Hering--Schenck--Smith~\cite{HSS} (see also Maclagan--Smith~\cite{MaclaganSmith}) provides an efficient way to exhibit many collections that give rise to closed immersions. If each $L_i$ on $X$ is the restriction of a basepoint-free line bundle on $\widetilde{X}$ then the morphism from Theorem~\ref{thm:1} is simply the restriction of that from \cite[Theorem~1]{CrawSmith}. This is rarely the case however, because the nef cone of $X$ is typically the union of the nef cones of a finite collection of ambient toric varieties.

Our second main result is more algebraic, and provides a fine moduli description of $X$ via the noncommutative algebra $A:=\kk Q/J_\mathscr{L}$. Since $\kk Q$ is constructed from sections of sheaves on the ambient toric variety rather than sections of sheaves on $X$, one does not expect a strong link between $X$ and the multigraded linear series until the geometry is influenced by the ideal of relations $J_\mathscr{L}$.  To this end,  the ideal $J_\mathscr{L}$ defines an ideal $I_\mathscr{L}$ in the Cox ring of $\vert\mathscr{L}\vert$, and the subscheme cut out by this ideal is the fine moduli space $\mathcal{M}_\vartheta(\modAL)$ of $\vartheta$-stable $A$-modules with dimension vector $(1, \dots, 1)$. This subscheme contains the image of the morphism $\varphi_{\vert \mathscr{L}\vert}$ from Theorem~\ref{thm:1}. This inclusion is proper in general, but by saturating $I_\mathscr{L}$ with the irrelevant ideal for the GIT quotient construction of $\vert\mathscr{L}\vert$ and by comparing the result with the ideal $I_Q$ that cuts out the image of $\varphi_{\vert \mathscr{L}\vert}$, we obtain the following (see 
Section~\ref{sec:theorem2}):

 \begin{theorem}
 \label{thm:2}
 For any Mori Dream Space $X$, there exist (many) collections $\mathscr{L}$ on $X$ such that the morphism $\varphi_{\vert\mathscr{L}\vert} \colon X \to \vert\mathscr{L}\vert$ identifies $X$ with the fine moduli space $\mathcal{M}_\vartheta(\modAL)$, and the tautological line bundles on $\mathcal{M}_\vartheta(\modAL)$ coincide with the line bundles of $\mathscr{L}$. 
 \end{theorem}
  
\noindent The analogous result for toric varieties \cite[Theorem~1.2]{CrawSmith} relies heavily on lattice-theoretic and combinatorial techniques that are unavailable for Mori Dream Spaces. Our proof of Theorem~\ref{thm:2} combines a careful analysis of paths in the quiver of sections together with several invocations of multigraded regularity. 

\smallskip

 More generally, whenever the morphism $\varphi_{\vert\mathscr{L}\vert} \colon X \to \vert\mathscr{L}\vert$ is a closed immersion it identifies $X$ with $\mathcal{M}_\vartheta(\modAL)$ when the saturation of $I_\mathscr{L}$ by the irrelevant ideal coincides with the ideal $I_Q$. These ideals can be computed explicitly in any given example, so it is possible to check directly whether Theorem~\ref{thm:2} holds (subject to computational limitations).  The resulting geometric quotient constructions of $X$ are new, and while it is impossible to improve upon the Hu--Keel construction of $X$  from the birational point of view, it is sometimes possible to encode more refined information of $X$ via $\mathscr{L}$, such as its bounded derived category of coherent sheaves

 As an application of our results we illustrate a quiver moduli construction that encodes the bounded derived category of coherent sheaves for a pair of del Pezzo surfaces in which $\mathscr{L}$ is chosen to be a full, strongly exceptional collection of basepoint-free line bundles, that is, the line bundles in  $\mathscr{L}$ freely generate the bounded derived category of coherent sheaves on $X$.  In each case, we deduce from Theorem~\ref{thm:1} and its corollaries that  $\varphi_{\vert\mathscr{L}\vert} \colon X \to \vert\mathscr{L}\vert$ is a closed immersion and, moreover,  we compute that the saturation of $I_\mathscr{L}$ by the irrelevant ideal coincides with the ideal $I_Q$ as in Theorem~\ref{thm:2}. It follows in each case that the del Pezzo surface $X$ is isomorphic to the fine moduli space $\mathcal{M}_\vartheta(\modAL)$, and the tautological bundles on the moduli space correspond under this isomorphism to the line bundles in $\mathscr{L}$. Until now it was known only that each del Pezzo surface $X$ is isomorphic to a connected component of a fine moduli space $\mathcal{M}_\vartheta(\modAL)$, see Bergman--Proudfoot~\cite{BergmanProudfoot}.
 
   \smallskip
  
We now describe the structure of the paper. Section~\ref{sec:2} defines the bound quiver of sections for collections of line bundles on a Mori Dream Space.  The morphism to the multigraded linear series and the geometric results leading to Theorem~\ref{thm:1} are introduced in Section~\ref{sec:3}, while the ideals of relations and the algebraic results leading to Theorem~\ref{thm:2} appear in Section~\ref{sec:4}. Finally, our application to tilting bundles on del Pezzo surfaces is presented in Section~\ref{sec:5}.
 
 \medskip

\noindent\textbf{Conventions.}
 Write $\kk$ for an algebraically closed field of characteristic zero, $\kk^\times$ for the one-dimensional algebraic torus over
 $\kk$, and $\NN$ for the semigroup of nonnegative integers. Given a vector space $V$, we write $\mathbb{P}^*(V)$ for the projective space of hyperplanes in $V$. For $d>0$ and $\mathbf{u}:=(u_1,\dots, u_d)\in \NN^d$ we adopt multiindex notation $x^{\mathbf{u}}:=\prod_{1\leq i\leq d} x_i^{u_i}\in \kk[x_1,\dots, x_d]$.
       
\medskip
\noindent\textbf{Acknowledgements.}  Special thanks to Tarig Abdelgadir and Arend Bayer for useful conversations, and to the anonymous referee for helpful comments. Thanks also to J\"{u}rgen Hausen, Markus Perling and Diane Maclagan. This work forms part of the second author's EPSRC-funded PhD thesis. The first author was supported in part by EPSRC grant EP/G004048.

 \section{Quivers of sections on Mori Dream Spaces}
 \label{sec:2}
 We begin by establishing our notation and introducing the bound quiver of sections for a collection of line bundles on a Mori Dream Space. These bound quivers encode the endomorphism algebra of the direct sum of the sheaves in the collection. 
 
 \subsection{Mori Dream Spaces}
 Let $X$ be a projective $\QQ$-factorial variety with divisor class group $\Cl(X)$ that is finitely generated and free of rank $\rho$. Let $D_1,\dots, D_\rho$ be Weil divisors whose classes provide an integral basis of $\Cl(X)$. Assume that the $\ZZ^\rho$-graded \emph{Cox ring}
 \[ 
 \Cox(X) := \bigoplus_{(m_{1},\ldots,m_{\rho})\in \ZZ^\rho} H^{0}\big(X, \mathscr{O}_X(m_1D_1+\dots + m_\rho D_\rho)\big)
 \]
 is a finitely generated $\kk$-algebra, where the multiplication is induced by multiplication of global sections. A different choice of integral basis for the divisor class group yields (non-canonically) an isomorphic Cox ring for $X$. Since $\Cl(X)$ is a finitely generated and free abelian group, $\Cox(X)$ is a unique factorisation domain by Elizondo--Kurano--Watanabe~\cite{EKW}. The main result of Hu--Keel~\cite[Proposition 2.9]{HuKeel} shows that finite generation of $\Cox(X)$ is equivalent to $X$ being a \emph{Mori Dream Space} (see \cite{HuKeel} for the definition that justifies the terminology). 
 
 Choose once and for all a presentation
 \begin{equation}
 \label{eqn:CoxX}
 \Cox(X) = \kk[x_{1},\ldots.x_{d}]/ I_{X}
 \end{equation}
 for some $d\in \NN$. The $\ZZ^\rho$-grading from $\Cox(X)$ can be lifted via the quotient map 
 \begin{equation}
 \label{eqn:HuKeelmap}
 \tau\colon   \kk[x_1,\dots, x_d]  \longrightarrow  \Cox(X).
 \end{equation}
to obtain a $\ZZ^\rho$-grading of $\kk[x_{1},\ldots,x_{d}]$ that makes $\tau$ a $\ZZ^\rho$-graded homomorphism of $\kk$-algebras. This defines a semigroup homomorphism $\deg\colon \NN^d\to \ZZ^\rho$ and an action of the algebraic torus $T:=\Hom(\ZZ^\rho, \kk^{\times})$ on $\mathbb{A}^{d}_{\kk}$.  Choose $\chi\in T^*=\ZZ^\rho$ such that every $\chi$-semistable point of $\mathbb{A}^d_\kk$ is $\chi$-stable and, moreover, such that a $\chi$-stable point exists (that is, we assume $\chi$ lies in an open GIT chamber for the $T$-action on $\mathbb{A}^d_\kk$).  Then the radical monomial ideal 
 \[
 B = \rad( x^{\mathbf{u}} \in \kk[x_1,\dots, x_d] : \deg(\mathbf{u})=m\chi \text{ for some } m>0)
 \]
cuts out the $\chi$-unstable locus in $\mathbb{A}^d_\kk$, and the geometric quotient $\mathbb{A}^d_\kk\git_\chi T := (\mathbb{A}^d_\kk\setminus \mathbb{V}(B))/T$ of the open subscheme of $\chi$-stable points in $\mathbb{A}^d_\kk$ by $T$ is a projective $\QQ$-factorial toric variety.

 The action of $T$ restricts to an action on $\Spec(\Cox(X))\cong \mathbb{V}(I_X)$ since $\tau$ is $\ZZ^\rho$-graded. Choose $\chi\in T^*$ that is ample on $X$ (so $\chi$-stable points exist), and assume that every $\chi$-semistable point of $\mathbb{A}^d_\kk$ is $\chi$-stable. The $\chi$-unstable locus in $\mathbb{V}(I_X)$ is cut out by the ideal $ B_X:= \rad(I_X+B)$. Since $\chi$ is ample on $X$, Hu--Keel~\cite[Proposition 2.11]{HuKeel} implies that the geometric quotient $\mathbb{V}(I_X)\git_\chi T:= (\mathbb{V}(I_X)\setminus \mathbb{V}(B_X))/T$ of the open subscheme of $\chi$-stable points by the action of $T$ is isomorphic to $X$. Since $\mathbb{V}(I_X)\subseteq \mathbb{A}^d_\kk$ is closed, there is a closed immersion
 \[
 \mathbb{V}(I_X)\git_\chi T\hookrightarrow \mathbb{A}^d_\kk\git_\chi T
 \]
 that embeds the Mori Dream Space $X$ into what we call the  \emph{ambient toric variety} $\widetilde{X}_\chi:=\mathbb{A}^d_\kk\git_\chi T$. In fact a choice is required here: the GIT chamber decomposition for the action of $T$ on $\mathbb{A}^d_\kk$ refines that for the action on $\Spec(\Cox(X))$, so the nef cone of $X$ is the union of nef cones of finitely many ambient toric varieties. Nevertheless,  any two such are isomorphic in codimension-one, and our construction requires only knowledge of the divisors in $\widetilde{X}_\chi$. Hereafter we choose one such ambient toric variety $\widetilde{X}:=\widetilde{X}_\chi$ and suppress the dependence on $\chi$.

 Every Mori Dream Space can be constructed from a bunched ring in the sense of Berchtold--Hausen~\cite[Definition~2.10, Theorem 4.2]{BerchtoldHausen}. It follows that the restriction of rank one reflexive sheaves from the ambient toric variety $\widetilde{X}$ to $X$ induces an isomorphism 
   \begin{equation}
  \label{eqn:psi}
  \psi \colon \Cl(\widetilde{X}) \longrightarrow \Cl(X)=\ZZ^\rho
  \end{equation}
 on divisor class groups, see \cite[Proposition~5.2]{BerchtoldHausen}.
     
 \begin{example}
 \label{exa:X4}
 Let $X_4$ be a del Pezzo surface obtained as the blow-up of $\mathbb{P}_\kk^2$ at four points in general position. The Picard group $\ZZ^5$ has a basis given by $H$, the pullback to $X_4$ of the hyperplane class on $\mathbb{P}^2_\kk$,  together with the four exceptional curves $R_1, R_2, R_3, R_4$. The semigroup homomorphism $\deg\colon \NN^{10}\to \ZZ^5$ obtained as multiplication by the matrix
  \[ \left[ \text{\footnotesize 
 $\begin{array}{rrrrrrrrrr}
  0 & 0 & 0 & 0 & 1 & 1 & 1 & 1 & 1 & 1 \\1 & 0 & 0 & 0 & -1 & -1 & -1 & 0 & 0 & 0 \\0 & 1 & 0 & 0 & -1 & 0 & 0 & -1 & -1 & 0 \\0 & 0 & 1 & 0 & 0 & -1 & 0 & -1 & 0 & -1 \\0 & 0 & 0 & 1 & 0 & 0 & -1 & 0 & -1 &-1
 \end{array}$
 } \right]
 \]   
induces the $\ZZ^5$-grading of $\kk[x_1,\dots, x_{10}]$, and the $T$-homogeneous ideal
 \[
 I_{X_{4}}:=\left(\begin{array}{c} x_{2}x_{5}-x_{3}x_{6}+x_{4}x_{7}, \;\; x_{1}x_{5}-x_{3}x_{8}+x_{4}x_{9} \\
  x_{1}x_{6}-x_{2}x_{8}+x_{4}x_{10}, \;\; x_{1}x_{7}-x_{2}x_{9}+x_{3}x_{10},\;\; x_{5}x_{10}-x_{6}x_{9}+x_{7}x_{8}\end{array}\right)
 \]
determines $\Cox(X_4)=\kk[x_1,\dots, x_{10}]/I_{X_4}$ following Batyrev--Popov~\cite{BatyrevPopov}. The ample linearisation $-K_{X_4} = 3H-R_1-R_2-R_3-R_4$ is a natural choice to construct an ambient toric variety,  but as Laface--Velasco~\cite[Example 2.11]{LafaceVelasco} note, this defines a non-$\QQ$-factorial toric quotient so it must lie in a GIT wall for the action of $T$ on $\mathbb{A}^{10}_\kk$.  However, the same authors also note that $\chi=11H-5R_1-3R_2-2R_3-R_4\in \ZZ^5$ is ample on $X$ and $\mathbb{A}^{10}\git_{\chi} T$ is $\QQ$-factorial so 
it lies in an open GIT chamber for the action of $T$ on $\mathbb{A}^{10}_\kk$. We therefore choose $\widetilde{X_4}:= \mathbb{A}^{10}\git_{\chi} T$.
\end{example}

 \subsection{Quivers and path algebras} 
 We now introduce notation for quivers and path algebras. 
 
 Let $Q$ be a finite connected quiver with vertex set $Q_0$, arrow set $Q_1$, and maps $\head, \tail \colon Q_1 \to Q_0$ indicating the head and tail of each arrow.  The characteristic functions $\chi_{i} \colon Q_0 \to \ZZ$ for $i \in Q_0$ and $\chi_{a} \colon Q_1 \to \ZZ$ for $a \in Q_1$ form the standard integral bases of the vertex space $\ZZ^{Q_0}$ and the arrow space $\ZZ^{Q_1}$ respectively.  A nontrivial path in $Q$ is a sequence of arrows $p = a_k \dotsb a_1$ with $\head(a_{j}) = \tail(a_{j+1})$ for $1 \leq j < k$.  Set $\tail(p) = \tail(a_{1}), \head(p) = \head(a_k)$ and $\supp(p)=\{a_1,\dots, a_k\}$.  Each $i \in Q_0$ gives a trivial path $e_i$ where $\tail(e_i) = \head(e_i) = i$.  The path algebra $\kk Q$ is the $\kk$-algebra whose underlying $\kk$-vector space has a basis of paths in $Q$, where the product of basis elements is the basis element defined by concatenation of the paths if possible, or zero otherwise.  A cycle is a path $p$ with $\tail(p) = \head(p)$, and a quiver is acyclic if every cycle is a trivial path $e_i$ for some vertex $i\in Q_0$.   A vertex $i\in Q_0$ is a source of $Q$ if it is not the head of an arrow. A spanning tree in $Q$ is a connected acyclic subquiver whose vertex set coincides with that of $Q$, and a spanning tree has root at $i\in Q_0$ if $i$ is the unique source of the tree.
   
   A relation in a quiver $Q$ is a $\kk$-linear combination of paths that share the same head and tail. Thus, each relation is of the form $\sum_{p\in \Gamma} c_p p$ for $c_p\in \kk$, where $\Gamma$ is a finite set of paths that share the same head and the same tail.  We do not assume that paths in $\Gamma$ traverse at least two arrows.  Any set of relations generates a two-sided ideal $J\subseteq \kk Q$, and the pair $(Q,J)$ is called a quiver with relations or, more simply, a \emph{bound quiver}. 
   
\subsection{Quivers of sections}  For $r\geq 0$, consider a collection of distinct line bundles
  \[
  \mathscr{L}:=(L_0,L_1,\dots,L_r)\subset \ZZ^\rho = \Cl(X)
  \]
  on the Mori Dream Space $X$, where $L_0=\mathscr{O}_X$ and $L_1, \dots, L_r$ are effective. For $0\leq i\leq r$, define $E_i\in \psi^{-1}(L_i)$ using the isomorphism $\psi\colon \Cl(\widetilde{X})\to \Cl(X)=\ZZ^\rho$ from \eqref{eqn:psi} to obtain a collection 
  \[
  \widetilde{\mathscr{L}}:= (E_0,E_1,\dots,E_r)
  \]
of distinct rank one reflexive sheaves on an ambient toric variety $\widetilde{X}$, where $E_0=\mathscr{O}_{\widetilde{X}}$. Since $\widetilde{X}$ is normal, each sheaf $E_i$ is of the form $\mathscr{O}_{\widetilde{X}}(D)$ for some $D\in \Cl(\widetilde{X})$. For $0\leq i\leq r$, define $D_i^\prime\in \Cl(\widetilde{X})$ such that $E_i=\mathcal{O}_{\widetilde{X}}(D_i^\prime)$. Writing $E^\vee:=\mathscr{H}om_{\mathscr{O}_{\widetilde{X}}}(E,\mathscr{O}_{\widetilde{X}})$ for the dual sheaf gives 
 \begin{equation}
 \label{eqn:doubledual}
 (E_j\otimes E_i^\vee)^{\vee\vee}\cong \mathcal{O}_{\widetilde{X}}(D_j^\prime-D_i^\prime)
 \end{equation}
for $0\leq i,j\leq r$, see for example Cox--Little--Schenck~\cite[Proposition 8.0.6]{CLS} (compare also Example~8.0.5 from \emph{loc.cit.}\ to see the necessity for using the double dual). For $0\leq i,j\leq r$, we say that a torus--invariant section $s \in H^0(\widetilde{X}, (E_j\otimes E_i^\vee)^{\vee\vee})$ is \emph{irreducible} if the section does not lie in the image of the multiplication map 
\begin{equation}
\label{eqn:multmapqos}
H^0\big(\widetilde{X},\mathcal{O}_{\widetilde{X}}(D_j^\prime-D_k^\prime)\big)\otimes_\kk H^0\big(\widetilde{X},\mathcal{O}_{\widetilde{X}}(D_k^\prime-D_i^\prime)\big)\longrightarrow H^0\big(\widetilde{X},\mathcal{O}_{\widetilde{X}}(D_j^\prime-D_i^\prime)\big)
\end{equation}
for any $k\neq i,j$, where we use isomorphism \eqref{eqn:doubledual}. The following extends the notion of a quiver of sections for a collection of line bundles on a projective toric variety due to Craw--Smith~\cite{CrawSmith}.
    
\begin{definition}
\label{def:qos}
The \emph{quiver of sections} of the collection $\mathscr{L}$ on $X$ is defined to be the quiver $Q$ with vertex set $Q_0 = \{ 0, \dotsc,r \}$, and where the arrows from $i$ to $j$ correspond to the irreducible sections in $H^0(\widetilde{X}, (E_j\otimes E_i^\vee)^{\vee\vee})$. 
\end{definition}

\begin{remark}
\label{rem:justification}
\begin{enumerate}
\item For $0\leq i,j\leq r$, the space of sections of the sheaf from \eqref{eqn:doubledual} is a subspace of the ring $\Cox(\widetilde{X}) = \kk[x_1,\dots, x_d]$, and the restriction of the map $\tau$ from \eqref{eqn:HuKeelmap} to this subspace is a surjective $\kk$-linear map
\[
\tau_{ij}\colon H^0\big(\widetilde{X}, \mathcal{O}_{\widetilde{X}}(D_j^\prime-D_i^\prime)\big)\longrightarrow H^0(X,L_j\otimes L_i^{-1})\cong\Hom_{\mathscr{O}_X}(L_i,L_j).
\]
Thus, the image under $\tau_{i,j}$ of the torus-invariant sections of $\mathcal{O}_{\widetilde{X}}(D_j^\prime-D_i^\prime)$ provide a spanning set for the space $\Hom_{\mathscr{O}_X}(L_i,L_j)$ of primary interest, see Proposition~\ref{prop:algebra} below. We abuse terminology by calling $Q$ the `quiver of sections of $\mathscr{L}$' even though paths in $Q$ from $i$ to $j$ are not constructed from a basis of $\Hom(L_i,L_j)$ as in the literature \cite{Craw, CrawSmith}.

\item In Definition~\ref{def:qos} we write the sheaf from \eqref{eqn:doubledual} as $(E_j\otimes E_i^\vee)^{\vee\vee}$ to express it directly in terms of the collection $\widetilde{\mathscr{L}}$. Of course, if each $E_i$ is invertible then we can write
\begin{equation}
\label{eqn:doubledualHom}
H^0(\widetilde{X}, (E_j\otimes E_i^\vee)^{\vee\vee})\cong \Hom_{\mathscr{O}_{\widetilde{X}}}(E_i,E_j)
\end{equation}
and $Q$ is the quiver of sections of the collection $\widetilde{\mathscr{L}}$ on $\widetilde{X}$ as defined in \cite{CrawSmith}. 
\item Composing arrows in $Q$ translates into multiplication of torus-invariant sections. That is, an arrow from $i$ to $k$ defined by a section of $(E_k\otimes E_i^\vee)^{\vee\vee}$ composed with an arrow from $k$ to $j$ defined by a section of $(E_j\otimes E_k^\vee)^{\vee\vee}$ gives the path from $i$ to $j$ defined by the product of these sections under the multiplication map \eqref{eqn:multmapqos}.
\item Definition~\ref{def:qos} depends a priori on the choice of ambient toric variety $\widetilde{X}$. However, given a presentation of the Cox ring as in \eqref{eqn:CoxX}, any two such ambient toric varieties are isomorphic in codimension-one, so $Q$ is independent of the choice of $\widetilde{X}$.
\end{enumerate}
\end{remark}     

\begin{lemma}
\label{lem:quiverprops}
The quiver of sections $Q$ is connected, acyclic, and $0\in Q_0$ is the unique source.
\end{lemma}
\begin{proof} 
We first suppose that $Q$ admits a nontrivial cycle, so there exists $i\neq j$ such that both $H^0((E_j\otimes E_i^\vee)^{\vee\vee})$ and $H^0((E_i\otimes E_j^\vee)^{\vee\vee})$ are nonzero. Remark~\ref{rem:justification}(3) implies that the product of these nonzero sections gives a nonconstant section of $\mathcal{O}_{\widetilde{X}}$, but $\widetilde{X}$ is projective, so $Q$ is acyclic after all. For $i\in Q_0$, the space $H^0((E_i\otimes E_0^\vee)^{\vee\vee})\cong H^0(E_i)$ has a torus-invariant element since $E_1,\dots, E_r$ effective and $E_0\cong \mathscr{O}_{\widetilde{X}}$, giving rise to a path in $Q$ from $0$ to $i\in Q_0$ as required.
\end{proof}

The quiver of sections depends purely on the collection of reflexive sheaves $\widetilde{\mathscr{L}}$ on $\widetilde{X}$, but we aim to encode information about the collection of line bundles $\mathscr{L}$ on the Mori Dream Space $X$.  To achieve this, write the Cox ring $\kk[x_1,\dots,x_d]$ of the toric variety $\widetilde{X}$ as the semigroup algebra of the semigroup $\NN^d$ of effective torus-invariant Weil divisors in $\widetilde{X}$. Define the \emph{label} of an arrow $a \in Q_1$, denoted $\div(a)\in \NN^{d}$, to be the divisor of zeroes in $\widetilde{X}$ of the defining torus-invariant section $s \in H^0(\widetilde{X}, (E_{\head(a)}\otimes E_{\tail(a)}^\vee)^{\vee\vee})$. More generally, the label of any path $p$ in $Q$ is the torus-invariant divisor $\div(p) := \sum_{a\in \supp(p)} \div(a)$. It is often convenient to consider the corresponding labelling monomial 
\[
x^{\div(p)}:= \prod_{a\in \supp(p)} x^{\div(a)}\in \kk[x_1,\dots, x_d]
\] 
in the Cox ring. Recall from \eqref{eqn:CoxX} that our chosen presentation of the Cox ring of $X$ determines an ideal $I_X\subset \kk[x_1,\dots, x_d]$.

\begin{definition}
\label{def:Jideal}
Consider the two-sided ideal 
   \[
  J_{\mathscr{L}}:= \Bigg( \sum_{p\in \Gamma} c_{p}p \in \kk Q \mid \begin{array}{c}  \exists \text{ finite set of paths } \Gamma \text{ in }Q \text{ that all have the same head and the} \\  \text{same tail, and } \exists\; \{c_p\in \kk : p\in \Gamma\} \text{ such that }\sum_{p\in \Gamma} c_{p} x^{\div(p)} \in I_{X}  \end{array}\Bigg)
 \]
 in the path algebra $\kk Q$. The pair $(Q,J_{\mathscr{L}})$ is the \emph{bound quiver of sections}  of the collection $\mathscr{L}$.
  \end{definition}

 \begin{proposition}
 \label{prop:algebra}
The quotient algebra $\kk Q/J_{\mathscr{L}}$ is isomorphic to $\End_{\mathscr{O}_X}\bigl( \bigoplus_{i\in Q_0} L_i \bigr)$, and each vertex $i\in Q_0$ satisfies $e_{i}( \kk Q / J_{\mathscr{L}}) e_{0} \cong H^{0}(X,L_{i})$.
 \end{proposition}
 \begin{proof}
 Adapt the proof of \cite[Proposition~3.3]{CrawSmith} slightly, replacing $\Hom(L_i,L_j)$ by $H^0(\widetilde{X}, (E_j\otimes E_i^\vee)^{\vee\vee})$ throughout, to see that the $\kk$-algebra epimorphism 
 \[
 \widetilde{\eta} \colon \kk Q\rightarrow\bigoplus_{i,j\in Q_0} H^0\big(\widetilde{X}, (E_j\otimes E_i^\vee)^{\vee\vee}\big)
 \]
sending $\sum c_{p}p$ to $\sum c_{p}x^{\div(p)}$ has kernel the following $T$-homogeneous ideal in $\kk Q$:
  \[
  J_{\widetilde{\mathscr{L}}}:= \big( p^+-p^- \in \kk Q \mid \head(p^+)=\head(p^-),
 \tail(p^+)=\tail(p^-), \div(p^+) = \div(p^-)\big)
 \]
For each $i, j\in Q_0$, isomorphism \eqref{eqn:doubledual} implies that the restriction of the $\ZZ^\rho$-graded homomorphism $\tau$ from \eqref{eqn:HuKeelmap} defines a surjective $\kk$-linear map $\tau_{ij}\colon H^0\big(\widetilde{X}, (E_j\otimes E_i^\vee)^{\vee\vee}\big)\to H^0(X,L_j\otimes L_i^{-1})$ with kernel the $\kk$-vector subspace $I_X\cap H^0\big(\widetilde{X}, (E_j\otimes E_i^\vee)^{\vee\vee}\big)$. The direct sum of all such maps defines the surjective, right-hand vertical map $\widehat{\tau}$ in the commutative diagram of $\kk$-algebras
  \begin{equation}
 \label{eqn:etadiagram}
  \begin{CD}   
  \kk Q  @>{\widetilde{\eta}}>>  \bigoplus_{i,j\in Q_0} H^0\big(\widetilde{X}, (E_j\otimes E_i^\vee)^{\vee\vee}\big) \\
  @|            @VV{\widehat{\tau}}V       \\
 \kk Q @>{\eta}>>  \End_{\mathscr{O}_{X}}\bigl( \bigoplus_{i\in Q_0} L_i \bigr) \\
 \end{CD}
 \end{equation}
where the epimorphism $\eta$ sends $\sum c_{p}p$ to the class $\sum c_{p}x^{\div(p)}\mod I_X$. The ideal $J_{\mathscr{L}}$ is the kernel of $\eta$, so the first statement holds.  The second statement follows from the first since we have $L_0= \mathscr{O}_X$ and we compose arrows and maps from right to left.  
\end{proof} 
 
 \section{Morphism to the multigraded linear series}
 \label{sec:3}
In this section we use the quiver of sections of a collection $\mathscr{L}$ of line bundles on a Mori Dream Space $X$ to define the corresponding multigraded linear series $\vert \mathscr{L}\vert$. Evaluating sections of line bundles defines a natural map from $X$ to $\vert\mathscr{L}\vert$, and we establish necessary and sufficient conditions that make it a morphism and a closed immersion.

\subsection{Toric quiver flag varieties}
Let $Q$ be a finite, connected quiver. A representation of $Q$ consists of a $\kk$-vector space
 $W_i$ for $i\in Q_0$ and a $\kk$-linear map $w_a \colon
 W_{\tail(a)} \to W_{\head(a)}$ for $a \in Q_1$.  It is convenient to write
 $W$ as shorthand for $\big((W_i)_{i\in Q_0},(w_a)_{a\in Q_1}\big)$.
 The dimension vector of $W$ is the vector $\underline{r}\in
 \ZZ^{Q_0}$ with components $r_i = \dim_\kk (W_i)$ for $i\in Q_0$. A map of representations $\psi\colon W\to W^\prime$ is a
 family $\psi_{i} \colon W_i^{\,} \to W_i^\prime$ of $\kk$-linear maps
 for $i\in Q_0$ satisfying $w_a^\prime \psi_{\tail(a)} =
 \psi_{\head(a)} w_a$ for $a \in Q_1$.  With composition defined
 componentwise, we obtain the abelian category of finite dimensional
 representations of $Q$. For $\theta \in \ZZ^{Q_0}$, define $\theta(W) := \sum_{0\leq
   i\leq \rho} \theta_i\dim_\kk(W_i)$.  Following King~\cite{King}, a representation $W$ of $Q$  is $\theta$-semistable if $\theta(W)=0$ and every
 subrepresentation $W^\prime \subset W$ satisfies
 $\theta(W^\prime)\geq 0$. Moreover, $W$ is $\theta$-stable if the
 only subrepresentations $W^\prime$ with $\theta(W^\prime)=0$ are 0
 and $W$.

 The incidence map $\inc \colon \ZZ^{Q_1} \to \ZZ^{Q_0}$ defined by setting $\inc(\chi_{a})=\chi_{\head(a)} - \chi_{\tail(a)}$ has image equal to the sublattice $\Wt(Q) \subset \ZZ^{Q_0}$ of functions $\theta \colon Q_0 \to \ZZ$ satisfying $\sum_{i \in Q_0} \theta_i = 0$. The vectors $\{\chi_i - \chi_0 : i\neq 0\}$ form a $\ZZ$-basis for $\Wt(Q)$.  The $\Wt(Q)$-grading of $\kk[y_a : a\in Q_1]$ determined by sending $y_a$ to $\inc(\chi_a)$ for $a\in Q_1$ induces a faithful action of the algebraic torus $G:=\Hom(\Wt(Q),\kk^\times)$ on $\mathbb{A}_\kk^{Q_1} = \Spec \kk[y_a : a\in Q_1]$ in which $g=(g_i)_{i\in Q_0}$ acts on $w=(w_a)_{a\in Q_1}$ as $(g \cdot w)_{a} = g_{\head(a)}^{\,} w_{a} g_{\tail(a)}^{-1}$. For $\theta\in \Wt(Q)$, let $\kk[y_a : a\in Q_1]_{\theta}$ denote the $\theta$-graded piece and 
\[
\mathbb{A}^{Q_1}_\kk\git_\theta G= \Proj\Big(\bigoplus_{j\geq 0} \kk[y_a : a\in Q_1]_{j\theta}\Big)
\]
the categorical quotient of the open subset of $\theta$-semistable points in $\mathbb{A}^{Q_1}_\kk$. 

Assume in addition that $Q$ is acyclic with a unique source $0\in Q_0$. The \emph{toric quiver flag variety} $Y_Q$ is the GIT quotient $\mathbb{A}^{Q_1}_\kk\git_\vartheta G$ linearised by the special weight $\vartheta:=\sum_{i\in Q_0} (\chi_i-\chi_0)\in \Wt(Q)$. Such varieties, studied initially by Craw--Smith~\cite{CrawSmith} and in greater generality by Craw~\cite{Craw}, can be characterised as follows:

 \begin{proposition}
 \label{prop:finemoduli}
 Let $Q$ be a finite, connected, acyclic quiver with a unique source $0\in Q_0$ and special weight $\vartheta=\sum_{i\in Q_0} (\chi_i-\chi_0)$. The toric quiver flag variety $Y_Q$ coincides with:
 \begin{enumerate}
 \item[\one] the GIT quotient $\mathbb{A}^{Q_1}_\kk\git_\vartheta G$ linearised by $\vartheta\in \Wt(Q)$;
 \item[\two] the geometric quotient of $\mathbb{A}^{Q_1}_\kk\setminus
   \mathbb{V}(B_{Q})$ by the action of $G$, where the irrelevant ideal is 
   \[
   B_{Q} := \Biggl( \prod\limits_{a \in \mathcal{T}} y_a : \text{$\mathcal{T}$ is a
      spanning tree of $Q$ rooted at $0$} \Biggl) = \bigcap\limits_{i \in Q_0\setminus\{0\}} \bigl( y_a : \head(a) = i \bigr) \, ;
    \]
 \item[\three] the fine moduli space
  $\mathcal{M}_\vartheta(Q)$ of $\vartheta$-stable
  representations of the quiver $Q$ of dimension vector $\underline{r}=(1,\dots, 1)\in \ZZ^{Q_0}$.
 \end{enumerate}
 Moreover, $Y_Q$ is a smooth projective toric variety obtained as a tower of projective space bundles.
  \end{proposition}
\begin{proof}
See Craw--Smith~\cite[Proposition~3.8]{CrawSmith} and Craw~\cite[Theorem~3.3]{Craw}.
\end{proof} 

The description of $Y_Q=\mathcal{M}_\vartheta(Q)$ as a fine moduli space of representations ensures that it carries a collection of tautological line bundles $\{\mathscr{W}_i : i\in Q_0\}$ with $\mathscr{W}_0\cong \mathscr{O}_{Y_Q}$ and  sheaf homomorphisms $\{\mathscr{W}_{\tail(a)}\to \mathscr{W}_{\head(a)} : a\in Q_1\}$ whose restriction to the fibre over $\mathcal{M}_\vartheta(Q)$ encodes the corresponding representation $\{W_{\tail(a)}\to W_{\head(a)} : a\in Q_1\}$. Moreover, the homomorphism of abelian groups $\Wt(Q)\to \Pic(Y_Q)$ satisfying
\begin{equation}
\label{rem:tautbundles}
(\theta_0,\dots, \theta_r)\mapsto \mathscr{W}_1^{\theta_1}\otimes \dots \otimes \mathscr{W}_r^{\theta_r}
\end{equation}
is an isomorphism. For more details, see \cite[Sections 2-3]{Craw}.

\subsection{Multigraded linear series} 
\label{sec:theorem1}
Let $\mathscr{L}=(\mathscr{O}_X,L_1,\dots, L_r)$ be a collection of distinct, effective line bundles on a Mori Dream Space $X$. Lemma~\ref{lem:quiverprops} guarantees that the corresponding quiver of sections $Q$ is finite, connected, acyclic and has a unique source $0\in Q_0$. 

\begin{definition}
The \emph{multigraded linear series} for $\mathscr{L}$ is the toric quiver flag variety $\vert \mathscr{L}\vert:= Y_Q$ of $Q$ from Proposition~\ref{prop:finemoduli}. It carries tautological line bundles $\{\mathscr{W}_i : i\in Q_0\}$ with $\mathscr{W}_0\cong \mathscr{O}_{Y_Q}$.
\end{definition}

\begin{remark}
\label{rem:justification2}
Just as $Q$ is not precisely the quiver of sections of $\mathscr{L}$ (see Remark~\ref{rem:justification}), it is perhaps an abuse of terminology to call $Y_Q$ the multigraded linear series of $\mathscr{L}$. Indeed, for the special case $\mathscr{L} = (\mathscr{O}_X, L_1)$ we have that $Y_Q\cong \mathbb{P}^*(H^0(E_1))$ is a projective space, but it need not coincide with the classical linear series $\vert L_1\vert$ because the epimorphism $\tau\vert_{H^0(\widetilde{X},E_i)} \colon H^0(\widetilde{X},E_1)\to H^0(X,L_1)$ from \eqref{eqn:HuKeelmap} need not be an isomorphism. 
\end{remark}     

In order to study morphisms from $X$ to the multigraded linear series $\vert\mathscr{L}\vert$, define 
\[
\widetilde{\Phi}\colon \kk[y_a : a\in Q_1] \rightarrow \kk[x_1,\dots,x_d]
\]
to be the $\kk$-algebra homomorphism sending $y_a$ to $x^{\div(a)}$ for $a\in Q_1$.  The actions of the groups $G=\Hom(\Wt(Q),\kk^*)$ and $T=\Hom(\ZZ^\rho,\kk^*)$ on $\kk[y_a : a\in Q_1]$ and $\kk[x_1,\dots,x_d]$ respectively arise from the horizontal semigroup homomorphisms in the diagram 
  \begin{equation}
 \label{eqn:equivariant}
  \begin{CD}   
  \NN^{Q_1}  @>{\inc}>>  \Wt(Q)\\
  @V{\div}VV            @VV{\pic}V       \\
 \NN^{d} @>{\deg}>>  \ZZ^\rho \\
 \end{CD}
 \end{equation}
 where the vertical maps satisfy $\div(\chi_a) = \div(a)$ for $a\in Q_1$ and $\pic(\chi_i)= E_i$ for $i\in Q_0$. The map $\widetilde{\Phi}$ is equivariant with respect to these actions precisely because \eqref{eqn:equivariant} commutes. Under the identification of $\Wt(Q)$ with the Picard group of $\vert \mathscr{L}\vert$, the subspace of the Cox ring $\kk[y_a : a\in Q_1]$ of $\vert\mathscr{L}\vert$ spanned by monomials of weight $\theta\in \Wt(Q)$ coincides with $H^0(\mathscr{W}_1^{\theta_1}\otimes\dots\otimes\mathscr{W}_r^{\theta_r})$. 
 
 Since the $T$-action on $\Cox(X)$ is compatible with that on $\kk[x_1,\dots,x_d]$ by \eqref{eqn:HuKeelmap}, the map
 \[
 \Phi:=\tau\circ \widetilde{\Phi}\colon \kk[y_a : a\in Q_1]\longrightarrow \Cox(X)
 \]
is equivariant. The induced equivariant morphism $\Phi^*\colon \mathbb{V}(I_X)\to \mathbb{A}^{Q_1}_\kk$ descends to a rational map $\varphi_{\vert\mathscr{L}\vert}\colon X\dashrightarrow \vert\mathscr{L}\vert$.  
 
\begin{proposition}
\label{prop:morphism}
Let $\mathscr{L}=(\mathscr{O}_X,L_1,\dots, L_r)$ be a collection of distinct, effective line bundles on $X$. The rational map $\varphi_{\vert\mathscr{L}\vert}\colon X\dashrightarrow \vert\mathscr{L}\vert$ is a morphism if and only if $L_i$ is basepoint-free for $1\leq i\leq r$.  \end{proposition}
\begin{proof}
For $x\in X$ choose any lift $\widetilde{x}\in \mathbb{V}(I_X)\setminus \mathbb{V}(B_X)$. The $G$-orbit of the quiver representation $\Phi^*(\widetilde{x})\in \mathbb{A}^{Q_1}_\kk$, which is independent of the choice of lift,  is obtained by evaluating the labels on arrows at $\widetilde{x}$, that is, by evaluating sections of the bundles $L_{\head(a)}\otimes L_{\tail(a)}^{-1}$ at $x$. The rational map $\varphi_{\vert\mathscr{L}\vert}\colon X\dashrightarrow \vert\mathscr{L}\vert$ is a morphism if and only if every such $\Phi^*(\widetilde{x})\in \mathbb{A}^{Q_1}_\kk$ is $\vartheta$-stable. Let $W^\prime=\big((W^\prime_i)_{i\in Q_0},(w^\prime_a)_{a\in Q_1}\big)$ be a proper subrepresentation of $\Phi^*(\widetilde{x})$. Since $\vartheta_0=-r$ and $\vartheta_i>0$ for $i>0$, the submodule $W^\prime$ of $\Phi^*(\widetilde{x})$ is $\vartheta$-destabilising if and only if $\dim_\kk(W^\prime_0 ) = 1$ and there exists $i>0$ such that for every path $p=a_\ell\cdots a_1$ from 0 to $i$, the composition $w^\prime_{a_\ell}\cdots w^\prime_{a_1}$ is the zero map. In particular, $\Phi^*(\widetilde{x})\in \mathbb{A}^{Q_1}_\kk$ is $\vartheta$-unstable if and only if there exists $i>0$ such that the evaluation of every section of $L_i$ at $x$ equals zero. Equivalently,  $\Phi^*(\widetilde{x})\in \mathbb{A}^{Q_1}_\kk$ is $\vartheta$-stable if and only if $L_i$ is basepoint-free for $1\leq i\leq r$.
  \end{proof}   
   
  The Cox ring of $X$ is a unique factorisation domain, so $\Ker(\Phi)$ is prime and hence so is the ideal 
\begin{equation}
\label{eqn:IQ}
I_{Q}:=\Big( f \in \kk[y_a : a\in Q_1] : f\in \Ker(\Phi) \text{ is }\Wt(Q)\text{-homogeneous} \Big)
\end{equation}
generated by its $\Wt(Q)$-homogeneous elements. This ideal can be computed explicitly as the kernel of the $\kk$-algebra homomorphism 
\begin{equation}
\label{eqn:Psibar}
\Psi\colon \kk[y_a : a\in Q_1]\to \Cox(X)\otimes_\kk \kk[\Wt(Q)]
\end{equation}
 satisfying $\Psi(y_a) = t_{\head(a)}x^{\div(a)}t_{\tail(a)}^{-1}$ for $a\in Q_1$; see Winn~\cite[Chapter~5]{Winn} for details. This ideal cuts out the image of the morphism constructed in Proposition~\ref{prop:morphism} as follows.
 
 \begin{proposition}
 \label{prop:morphismbundles}
Let $\mathscr{L}=(\mathscr{O}_X,L_1,\dots, L_r)$ be a collection of distinct, basepoint-free line bundles on $X$ with quiver of sections $Q$. Then
\begin{enumerate}
\item[\one] the image of the morphism $\varphi_{\vert\mathscr{L}\vert}\colon X\to \vert\mathscr{L}\vert$ is $\mathbb{V}(I_Q)\git_\vartheta G$; and
\item[\two] the tautological line bundles on $\vert \mathscr{L}\vert$ satisfy $\varphi_{\vert \mathscr{L}\vert}^*(\mathscr{W}_i) = L_i$ for $i\in Q_0$.
\end{enumerate}
 \end{proposition}
 \begin{proof}
Since $X$ is complete, the image of $\varphi_{\vert \mathscr{L}\vert}$ is a closed subscheme of $\vert\mathscr{L}\vert$. The geometric quotient construction of $\vert\mathscr{L}\vert$ from Proposition~\ref{prop:finemoduli}(i) implies that the image is therefore the geometric  quotient of a $G$-invariant closed subscheme of $\mathbb{A}^{Q_1}_\kk\setminus \mathbb{V}(B_Q)$.  The affine  variety $\mathbb{V}(\Ker(\Phi))$ is the image of the equivariant morphism $\Spec(\Cox(X))\to \mathbb{A}^{Q_1}_\kk$ induced by $\Phi$, and the variety $\mathbb{V}(I_Q)$ cut out by the $\Wt(Q)$-homogeneous part of $\Ker(\Phi)$ is the minimal $G$-invariant algebraic set in $\mathbb{A}^{Q_1}$ containing all $G$-orbits intersecting $\mathbb{V}(\Ker(\Phi))$. The image of $\varphi_{\vert \mathscr{L}\vert}$ is therefore the geometric quotient of $\mathbb{V}(I_Q)\setminus \mathbb{V}(B_Q)$ by the action of $G$. This coincides with the GIT quotient $\mathbb{V}(I_Q)\git_\vartheta G$ by Proposition~\ref{prop:finemoduli}, so \one\ holds.   For part \two, the tautological bundle $\mathscr{W}_i$ on $\vert\mathscr{L}\vert$ corresponds to the weight $\chi_i-\chi_0\in \Wt(Q)$ under the isomorphism \eqref{rem:tautbundles}. Since the equivariant morphism $\Spec(\Cox(X))\to \mathbb{A}^{Q_1}_\kk$ factors through $\mathbb{A}^d_\kk$, examining \eqref{eqn:HuKeelmap} and diagram \eqref{eqn:equivariant} shows that $\varphi_{\vert \mathscr{L}\vert}^*(\mathscr{W}_i) = (\psi\circ \pic)(\chi_i-\chi_0) = \psi(E_i) = L_i$ for $i\in Q_0$.
 \end{proof}

\begin{proof}[Proof of Theorem~\ref{thm:1}]
 Proposition~\ref{prop:morphism} establishes that $\varphi_{\vert \mathscr{L}\vert}\colon X\dashrightarrow \vert \mathscr{L}\vert$ is a morphism if and only if $L_i$ is basepoint-free for $1\leq i\leq r$. Proposition~\ref{prop:morphismbundles} then  presents the image explicitly as a geometric quotient, and establishes that the tautological line bundles on $\vert \mathscr{L}\vert$ satisfy $\varphi_{\vert \mathscr{L}\vert}^*(\mathscr{W}_i) = L_i$ for $i\in Q_0$ as required.
\end{proof}

\begin{remark}
\label{rem:piggyback}
The list of reflexive sheaves $\widetilde{\mathscr{L}}$ on $\widetilde{X}$ determines the ideal
\begin{equation}
\label{eqn:IQtilde}
I_{\widetilde{Q}} =\Big( f \in \kk[y_a : a\in Q_1] : f\in \Ker(\widetilde{\Phi}) \text{ is }\Wt(Q)\text{-homogeneous} \Big)
\end{equation}
obtained as the toric ideal of the semigroup homomorphism $\inc\oplus\div \colon \mathbb{N}^{Q_1}\to \Wt(Q)\oplus \NN^d$. If each reflexive sheaf in $\widetilde{\mathscr{L}}$ is a basepoint-free line bundle on $\widetilde{X}$, then \cite[Theorem~1]{CrawSmith} gives a morphism $\varphi_{\vert \widetilde{\mathscr{L}}\vert}\colon \widetilde{X}\to\mathbb{V}(I_{\widetilde{Q}})\git_\vartheta G$ whose restriction to $X$ is the morphism $\varphi_{\vert\mathscr{L}\vert}\colon X\to \mathbb{V}(I_Q)\git_\vartheta G$ from Proposition~\ref{prop:morphismbundles}. However, this is typically not the case as in Example~\ref{exa:X4part2}.
\end{remark}

\subsection{Criteria for closed immersion}
A collection $\mathscr{L}$ is said to be \emph{very ample} if the morphism $\varphi_{\vert\mathscr{L}\vert}$ from Proposition~\ref{prop:morphism} is a closed immersion. We now introduce a necessary and sufficient condition for $\mathscr{L}$ to be very ample.  
We (enhance and) adapt the proofs of \cite[Proposition~5.7]{Craw} and \cite[Corollary~4.10]{CrawSmith} to our situation because $Q$ is not precisely the quiver of sections for $\mathscr{L}$ (see Remarks~\ref{rem:justification} and \ref{rem:justification2}).  

\begin{theorem}
\label{thm:closedimmersion}
Let $\mathscr{L}=(\mathscr{O}_X,L_1,\dots, L_r)$ be a collection of distinct, basepoint-free line bundles on $X$. The following are equivalent:
\begin{enumerate}
\item[\one] the morphism  $\varphi_{\vert\mathscr{L}\vert}\colon X\rightarrow \vert\mathscr{L}\vert$ is a closed immersion;
\item[\two]  the image of the multiplication map 
 \begin{equation}
 \label{eqn:multmap}
 H^0(L_1)\otimes \dots \otimes H^0(L_r)\longrightarrow H^0(L_1\otimes \cdots \otimes L_r).
\end{equation}
 is a very ample linear series;
 \item[\three] the map $\prod_{1\leq i\leq r} \varphi_{\vert L_i\vert}\colon X\to \vert L_1\vert\times \dots \times \vert L_r\vert$ is a closed immersion.
\end{enumerate}
\end{theorem}
\begin{proof}
 The toric variety $\vert\mathscr{L}\vert$ is smooth, so the ample bundle $\vartheta=\mathscr{W}_1\otimes \dots\otimes \mathscr{W}_r$ determines the closed immersion $\varphi_{\vert\vartheta\vert}\colon \vert \mathscr{L}\vert \longrightarrow \mathbb{P}^*\big(H^0(\vartheta)\big)$. The composition $\varphi_{\vert \vartheta\vert}\circ \varphi_{\vert \mathscr{L}\vert}\colon X\to \mathbb{P}^*(H^0(\vartheta))$ is determined by the 
line bundle $(\varphi_{\vert \vartheta\vert}\circ \varphi_{\vert \mathscr{L}\vert})^*(\vartheta)= (\psi\circ \pic)(\theta) = L_1\otimes \dots \otimes L_r$ and the subspace of sections $\Phi(H^0(\vartheta))\subseteq H^0(L_1\otimes \dots \otimes L_r)$. We claim that $\Phi(H^0(\vartheta))$ coincides with the image $V$ of the multiplication map \eqref{eqn:multmap}, in which case  $\varphi_{\vert \vartheta\vert}\circ \varphi_{\vert \mathscr{L}\vert}$ coincides with the (a priori rational) map $\varphi_V\colon X\to\mathbb{P}^*(V)$ to the classical linear series. Indeed, for $\theta=(\theta_0,\dots,\theta_r)\in \Wt(Q)$, the restriction of $\Phi$ to the subspace spanned by monomials of weight $\theta$ defines a $\kk$-linear map
   \[
 \Phi_\theta \colon H^0(\mathscr{W}_1^{\theta_1}\otimes\dots\otimes\mathscr{W}_r^{\theta_r})\to H^0(L_1^{\theta_1}\otimes\dots\otimes L_r^{\theta_r})
 \]
because $(\psi\circ \pic)(\theta) = L_1^{\theta_1}\otimes \dots \otimes L_r^{\theta_r}$.  In particular, the map $\Phi_\vartheta$ for $\vartheta=\sum_{1\leq i\leq r} (\chi_i-\chi_0)$ and the product $\otimes_{1\leq i\leq r} \Phi_{(\chi_i-\chi_0)}$ fit in to a commutative diagram of $\kk$-vector spaces
   \begin{equation}
   \label{eqn:WtoL}
   \begin{CD}
   H^0(\mathscr{W}_1)\otimes \dots \otimes H^0(\mathscr{W}_r)@>>> H^0(\mathscr{W}_1\otimes \dots\otimes \mathscr{W}_r)  \\
    @V{\bigotimes_{1\leq i\leq r} \Phi_{(\chi_i-\chi_0)}}VV  @VV{\Phi_\vartheta}V \\
   H^0(L_1)\otimes \dots \otimes H^0(L_r) @>>>   H^0(L_1\otimes \cdots \otimes L_r)
   \end{CD}
 \end{equation}
in which the horizontal maps are given by multiplication. For $1\leq i\leq r$, the map $\Phi_{(\chi_i-\chi_0)}$ can be obtained by composing three surjective maps, namely, the isomorphism $H^0(\mathscr{W}_i) \to e_i(\kk Q)e_0$ from \cite[Corollary~3.5]{Craw}, the restriction to $e_i(\kk Q)e_0$ of the epimorphism $\widetilde{\eta}\colon \kk Q\to \End(\bigoplus_{i\in Q_0} E_i)$ from the proof of Proposition~\ref{prop:algebra}, and the restricted map $\tau\vert_{H^0(E_i)}\colon H^0(E_i)\to H^0(L_i)$ from \eqref{eqn:HuKeelmap}. It follows that $\bigotimes_{1\leq i\leq r} \Phi_{(\chi_i-\chi_0)}$ is surjective. Every monomial of weight $\vartheta$ in $\kk[y_a : a\in Q_1]$ can be decomposed as a product of monomials of weight $\chi_i-\chi_0\in \Wt(Q)$ for $i\in Q_0$, so the top horizontal map in diagram \eqref{eqn:WtoL} is surjective. Commutativity of the diagram then implies that the image of $\Phi_\vartheta$ coincides with the image $V$ of \eqref{eqn:multmap}. This proves the claim. 

Since  $V$ is the image of the multiplication map \eqref{eqn:multmap}, the morphism $\varphi_V\colon X\rightarrow \mathbb{P}^*(V)$  is the composition of the product $\prod_{1\leq i\leq r} \varphi_{\vert L_i\vert}\colon X\longrightarrow \vert L_1\vert\times \dots \times \vert L_r\vert$ of morphisms to the classical linear series and the appropriate Segre embedding to $\mathbb{P}^*(V)$. The claim implies that the diagram
\medskip
\[
\setlength{\arraycolsep}{0.8cm} 
\begin{array}{ccc}
\Rnode{a}{\vert L_1\vert\times \dots \times \vert L_r\vert} & \Rnode{b}{\mathbb{P}^*(V)} & \Rnode{c}{\mathbb{P}^*(H^0\big(\vartheta)\big)}\\[1.2cm]
\Rnode{d}{X} & \Rnode{e}{} & \Rnode{f}{\vert \mathscr{L}\vert} \end{array} 
 \psset{nodesep=5pt,arrows=->,linewidth=0.5pt} 
 \everypsbox{\scriptstyle}
 \ncline{a}{b}\Aput{\text{Segre}} 
 \ncline{b}{c}\Aput{\iota} 
 \ncline{d}{a}\Aput{\prod_{1\leq i\leq r} \varphi_{\vert L_i\vert}} 
 \ncline{d}{f}\Aput{\varphi_{\vert \mathscr{L}\vert}}
 \ncline{f}{c}\Bput{\varphi_{\vert \vartheta\vert}}
\]
commutes, where $\iota$ is the closed immersion of projective spaces induced by $\Phi_\vartheta$.  Three maps in the diagram are closed immersions, so $\varphi_{\vert\mathscr{L}\vert}$ is a closed immersion if and only if $\prod_{1\leq i\leq r} \varphi_{\vert L_i\vert}$ is a closed immersion if and only if the linear series $V$ is very ample as required.
\end{proof}    
    
\begin{remark}
Neither of the maps from statements \one\ and \three\ of Theorem~\ref{thm:closedimmersion} factors through the other. Typically $\vert\mathscr{L}\vert$ has much lower dimension than $\vert L_1\vert\times \dots \times \vert L_r\vert$, so the multigraded linear series is a more efficient multigraded ambient space than the product. 
\end{remark}

\begin{corollary}
 \label{coro:veryample}
 Let $L_1,\dotsc,L_{r-1}$ be distinct, basepoint-free line bundles on $X$.  If
 the subsemigroup of $\Pic(X)$ generated by $L_1, \dotsc, L_{r-1}$
 contains an ample bundle, then there exists a line bundle $L_r$
 such that the quiver of sections for $\mathscr{L} =
 (\mathscr{O}_X, L_1, \dotsc, L_r)$ is very ample.
\end{corollary}
\begin{proof}
Theorem~\ref{thm:closedimmersion} implies that $\varphi_{\vert\mathscr{L}\vert}$ is a closed immersion if $L_1\otimes\dots\otimes L_r$ is very ample and the map \eqref{eqn:multmap} is surjective. The proof of \cite[Proposition 4.14]{CrawSmith} now applies verbatim.
\end{proof}

 \begin{example}
 \label{exa:X4part2}
 Continuing Example~\ref{exa:X4}, let $X_4$ be the del Pezzo surface for which the ample linearisation $\chi=11H-5R_1-3R_2-2R_3-R_4$ defines $\widetilde{X_4}:= \mathbb{A}^{10}\git_{\chi} T$. We compute using the intersection pairing that each line bundle in the list 
\begin{equation}
\label{eqn:fsecX4}
\mathscr{L}=(\mathscr{O}_{X_4}, H, 2H-R_1, 2H-R_2, 2H-R_3, 2H-R_4, 2H)
\end{equation}
is basepoint-free but not ample. Write $\widetilde{\mathscr{L}}=(E_0, E_1, \dots, E_6)$. Since each $E_i$ is basepoint-free on some ambient toric variety, the code from \cite[Example~2.11]{LafaceVelasco} computes the irrelevant ideal for the GIT quotient $\mathbb{A}^d_\kk\git_{E_i} T$ determined by the corresponding linearisation $E_i\in \ZZ^5$.  By comparing each with the irrelevant ideal of $\chi\in \ZZ^5$ we see that $E_3, E_4, E_5$ are not basepoint-free line bundles on $\widetilde{X}_4$. In particular, we cannot deduce that $\varphi_{\vert \mathscr{L}\vert}$ is a morphism simply by restriction from the toric case (compare Remark~\ref{rem:piggyback}), though it is a morphism by Proposition~\ref{prop:morphism}.

To investigate $\varphi_{\vert \mathscr{L}\vert}$ directly in this case, the quiver of sections $Q$ is shown in Figure~\ref{fig:X4easy},  where each arrow is labelled by the torus-invariant section of the relevant reflexive sheaf on $\widetilde{X_4}$. 
      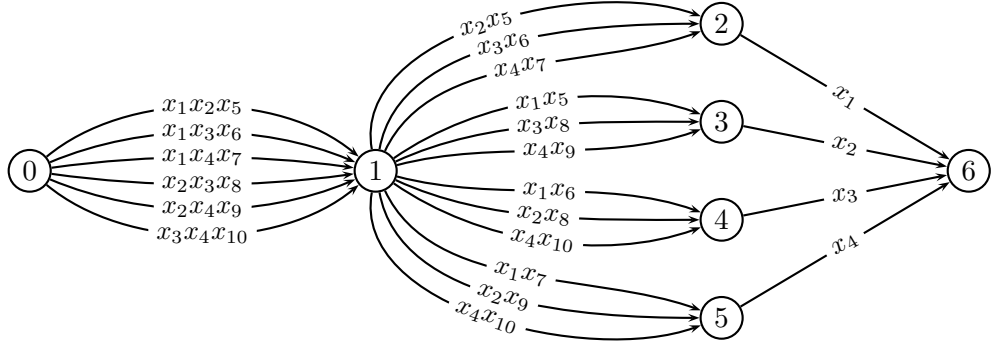
\begin{figure}[!ht]
    \centering
      \psset{unit=1.3cm}
      \begin{pspicture}(-0.6,-0.3)(10,3.2)
      \cnodeput(0,1.5){AA}{0}
      \cnodeput(3.5,1.5){A}{1}
      \cnodeput(7,3){B}{2} 
      \cnodeput(7,2){C}{3}
       \cnodeput(7,1){D}{4}
       \cnodeput(7,0){E}{5} 
      \cnodeput(9.5,1.5){F}{6}
      \psset{nodesep=0pt}
        \nccurve[angleA=40,angleB=140]{->}{AA}{A}\lput*{:U}{\small{$x_1x_2x_5$}}
        \nccurve[angleA=23,angleB=157]{->}{AA}{A}\lput*{:U}{\small{$x_1x_3x_6$}}
        \nccurve[angleA=8,angleB=172]{->}{AA}{A}\lput*{:U}{\small{$x_1x_4x_7$}}
         \nccurve[angleA=352,angleB=188]{->}{AA}{A}\lput*{:U}{\small{$x_2x_3x_8$}}
        \nccurve[angleA=337,angleB=203]{->}{AA}{A}\lput*{:U}{\small{$x_2x_4x_9$}}
        \nccurve[angleA=320,angleB=220]{->}{AA}{A}\lput*{:U}{\small{$x_3x_4x_{10}$}}
        \nccurve[angleA=100,angleB=160]{->}{A}{B}\lput*{:U}{\small{$x_2x_5$}}
        \nccurve[angleA=80,angleB=180]{->}{A}{B}\lput*{:U}{\small{$x_3x_6$}}
        \nccurve[angleA=65,angleB=200]{->}{A}{B}\lput*{:U}{\small{$x_4x_7$}}
         \nccurve[angleA=35,angleB=160]{->}{A}{C}\lput*{:U}{\small{$x_1x_5$}}
        \nccurve[angleA=25,angleB=180]{->}{A}{C}\lput*{:U}{\small{$x_3x_8$}}
        \nccurve[angleA=15,angleB=200]{->}{A}{C}\lput*{:U}{\small{$x_4x_9$}}
          \nccurve[angleA=345,angleB=160]{->}{A}{D}\lput*{:U}{\small{$x_1x_6$}}
        \nccurve[angleA=335,angleB=180]{->}{A}{D}\lput*{:U}{\small{$x_2x_8$}}
        \nccurve[angleA=325,angleB=200]{->}{A}{D}\lput*{:U}{\small{$x_4x_{10}$}}
         \nccurve[angleA=295,angleB=160]{->}{A}{E}\lput*{:U}{\small{$x_1x_7$}}
        \nccurve[angleA=280,angleB=180]{->}{A}{E}\lput*{:U}{\small{$x_2x_9$}}
        \nccurve[angleA=260,angleB=200]{->}{A}{E}\lput*{:U}{\small{$x_4x_{10}$}}
        \ncline{->}{B}{F}\lput*{:U}{\small{$x_1$}}    
        \ncline{->}{C}{F}\lput*{:U}{\small{$x_2$}}    
        \ncline{->}{D}{F}\lput*{:U}{\small{$x_3$}}    
        \ncline{->}{E}{F}\lput*{:U}{\small{$x_4$}}    
               \end{pspicture}    \caption{A quiver of sections for a collection on $X_4$}
  \label{fig:X4easy}
  \end{figure} 
 Arrows with tail at 0 are listed $a_1,\dots, a_6$ from the top of Figure~\ref{fig:X4easy} to the bottom;  list those with tail at 1 as $a_7, \dots, a_{18}$ from the top of the figure to the bottom; and list those with head at 6 as $a_{19}, \dots, a_{22}$ from the top to the bottom. Likewise, list the coordinates of $\mathbb{A}^{Q_1}_\kk$ as $y_1,\dots, y_{22}$, and compute the kernel of \eqref{eqn:Psibar} to obtain the ideal
\[
I_Q = \left(\begin{array}{c} 
y_{16}-y_{17}+y_{18},\; y_{13}-y_{14}+y_{15}, \; y_{10}-y_{11}+y_{12}, \; y_7-y_8+y_9,\;  y_3-y_5+y_6\\
y_2-y_4+y_6,\; y_1-y_4+y_5,\; 
y_{15}y_{21}-y_{18}y_{22},\; y_{12}y_{20}-y_{17}y_{22},\; y_{11}y_{20}-y_{14}y_{21} \\ 
y_9y_{19}-y_{17}y_{22}+y_{18}y_{22},\: y_8y_{19}-y_{14}y_{21}+y_{18}y_{22},\; y_6y_{17}-y_5y_{18}, y_6y_{14}-y_4y_{15} \\
y_5y_{11}-y_4y_{12},\; y_5y_8-y_6y_8-y_4y_9+y_6y_9,\; 

y_8y_{15}y_{17}-y_9y_{14}y_{18}-y_8y_{15}y_{18}+y_9y_{15}y_{18}\\
y_{11}y_{15}y_{17}-y_{12}y_{14}y_{18},  \; y_9y_{11}y_{17}-y_8y_{12}y_{17}+y_8y_{12}y_{18}-y_9y_{12}y_{18}\\
y_9y_{11}y_{14}-y_8y_{12}y_{14}+y_8y_{11}y_{15}-y_9y_{11}y_{15} 
    \end{array}\right)
 \]
 that cuts out the image of $\varphi_{\vert\mathscr{L}\vert}\colon X_4\to \vert\mathscr{L}\vert$. 
 
 We claim that $\varphi_{\vert\mathscr{L}\vert}$ is a closed immersion, and hence $X_4\cong \mathbb{V}(I_Q)\git_\vartheta G$. Indeed, for $1\leq i\leq 4$ we have $L_{i+1}=2H-R_{i}$, and the intersection pairing shows that $\varphi_{\vert L_{i+1}\vert}\colon X_4\to \mathbb{P}_{\mathbb{P}^1}(\mathscr{O}_{\mathbb{P}^1}\oplus \mathscr{O}_{\mathbb{P}^1}(1))$ contracts the $(-1)$-curves $\{R_j : j\neq i\}$ but not $R_i$. A simple case-by-case analysis shows that the morphism $\prod_{2\leq i\leq 5} \varphi_{\vert L_i\vert}$ separates all points and tangent vectors of $X_4$ and hence so does $\prod_{1\leq i\leq 6} \varphi_{\vert L_i\vert}$. We deduce from Theorem~\ref{thm:closedimmersion} that $\varphi_{\vert\mathscr{L}\vert}\colon X_4\to \vert\mathscr{L}\vert$ is a closed immersion.
   \end{example}

 \section{Fine moduli of bound quiver representations}
 \label{sec:4}
 This section establishes when the morphism $\varphi_{\vert\mathscr{L}\vert}\colon X\to \vert \mathscr{L}\vert$ induces an isomorphism between the Mori Dream Space $X$ and a fine moduli space $\mathcal{M}_\vartheta(\modAL)$ of $\vartheta$-stable modules over the algebra $A_\mathscr{L}=\End(\bigoplus_{i\in Q_0} L_i)$. Our main algebraic result is an efficient construction for collections of line bundles that induce the isomorphism.
   
\subsection{Bound quiver representations}
Let $Q$ be a quiver. For any nontrivial path $p = a_k \cdots a_1$ in $Q$, define the monomial $y_p:=y_{a_k}\cdots y_{a_1}\in \kk[y_a : a\in Q_1]$, and for any representation $W$ of $Q$, define $w_p \colon W_{\tail(p)} \to W_{\head(p)}$ to be the $\kk$-linear map $w_p = w_{a_k} \dotsb w_{a_1}$ obtained by composition. Let $J\subset \kk Q$ be a two-sided ideal of relations with generators of the form $\sum_{p\in \Gamma} c_pp$, where each $\Gamma$ is a finite set of paths that share the same head and the same tail. A representation $W$ of $Q$ is a representation of the bound quiver $(Q,J)$ if and only if $ \sum_{p\in \Gamma} c_{p}w_p =0$ for each $\Gamma$ arising in the definition of $J$.  A point in representation space $(w_a)\in\mathbb{A}_\kk^{Q_1}$ defines a representation of $(Q,J)$ if and only it lies in the subscheme
$\mathbb{V}(I_J)$ cut out by the ideal 
 \[
 I_J:=  \Bigg( \sum_{p\in \Gamma} c_{p}y_p \in \kk[y_a : a\in Q_1]  \mid \sum_{p\in \Gamma} c_pp \in J\Bigg)
 \]
 of relations in $\kk[y_a : a\in Q_1]$. The ideal $I_J$ is $\Wt(Q)$-homogeneous, so $\mathbb{V}(I_J)$ is $G$-invariant and the GIT quotient
\begin{equation}
\mathcal{M}_\vartheta(Q,J):= \mathbb{V}(I_J)\git_\vartheta G =
\textstyle{ \Proj \Big{(}\bigoplus_{j \geq 0} \big(\kk[y_a : a\in Q_1]/I_J)_{j \vartheta} \Big{)}}
\end{equation}
is the fine moduli space of $\vartheta$-stable representations of $(Q,J)$ with dimension vector $(1,\dots, 1)$. The tautological bundles on $\mathcal{M}_\vartheta(Q,J)$ are obtained from those on $\mathcal{M}_\vartheta(Q)$ by restriction.
 
 \begin{remark}
 \label{rem:modules}
 The abelian category of finite-dimensional representations of $(Q,J)$ is equivalent to the category of finitely-generated $A:=\kk Q/J$-modules, so $\mathcal{M}_\vartheta(Q,J)$ is equivalently the fine moduli space of $\vartheta$-stable $A$-modules that are isomorphic as $\bigl(
\bigoplus_{i \in Q_0} \kk e_i \bigr)$-modules to $\bigoplus_{i \in Q_0} \kk e_i$.
\end{remark}

\subsection{Ideals of relations}
A list $\mathscr{L}$ of line bundles on $X$ defines a pair of two-sided ideals in $\kk Q$ and hence a pair of ideals of relations in $\kk[y_a : a\in Q_1]$. First, the ideal $J_\mathscr{L}$ from Definition~\ref{def:Jideal} determines the ideal of relations $I_\mathscr{L}:=  I_{(J_{\mathscr{L}})}$, that is, the ideal
 \begin{equation}
 \label{eqn:IL}
 I_{\mathscr{L}}:= \Bigg( \sum_{p\in \Gamma} c_{p}y_p \in \kk[y_a : a\in Q_1]  \mid \begin{array}{c}   \exists \text{ finite set of paths } \Gamma \text{ with same head and same} \\  \text{tail, and } \exists\; c_p\in \kk\text{ such that }\widetilde{\Phi}\big(\sum_{p\in \Gamma} c_{p} y_p\big) \in I_{X}\end{array}\!\!\Bigg).
 \end{equation}
Each generator of $I_{\mathscr{L}}$ is $\Wt(Q)$-homogeneous and lies in $\Ker(\Phi)$, so $I_{\mathscr{L}}$ is contained in the prime ideal $I_Q$ from \eqref{eqn:IQ}.  Winn~\cite[Chapter~5]{Winn} presents code that calculates $I_{\mathscr{L}}$ explicitly. 

 In addition, the kernel $J_{\widetilde{\mathscr{L}}}$ of the epimorphism $\widetilde{\eta}\colon \kk Q \to \End_{\mathscr{O}_{\widetilde{X}}}(\bigoplus_{i\in Q_0} E_i)$ from the proof of Proposition~\ref{prop:algebra} determines the noncommutative toric ideal of relations $I_{\widetilde{\mathscr{L}}}:=  I_{(J_{\widetilde{\mathscr{L}}})}$, where
 \begin{equation}
 \label{eqn:ILtilde}
 I_{\widetilde{\mathscr{L}}}:=  \Bigg( \sum_{p\in \widetilde{\Gamma}} c_{p}y_p \in \kk[y_a : a\in Q_1]  \mid \begin{array}{c}    \exists \text{ finite set of paths } \widetilde{\Gamma} \text{ with same head and same} \\
  \text{tail, and } \exists\; c_p\in \kk\text{ such that }\widetilde{\Phi}\big(\sum_{p\in \widetilde{\Gamma}} c_{p} y_p\big) =0\end{array}\Bigg).
 \end{equation}
We have that $I_{\widetilde{\mathscr{L}}}$ is contained both in the ideal of equations $I_{\widetilde{Q}}$ from \eqref{eqn:IQtilde}, and in $I_\mathscr{L}$ from \eqref{eqn:IL}. Compute the affine varieties in $\mathbb{A}^{Q_1}_\kk$ cut out by the ideals $I_{\widetilde{\mathscr{L}}}, I_{\widetilde{Q}}, I_\mathscr{L}, I_Q\subset \kk[y_a : a\in Q_1]$,  remove from each the $\vartheta$-unstable locus $\mathbb{V}(B_Q)$, and compute the geometric quotient by the action of $G$ to obtain the left-hand square in the commutative diagram of GIT quotients 
  \begin{equation}
   \label{eqn:GITquotients}
   \begin{CD}
  \mathbb{V}(I_{\widetilde{Q}})\git_\vartheta G @>>>  \mathcal{M}_\vartheta(Q,J_{\widetilde{\mathscr{L}}})@>>>  \mathbb{A}^{Q_1}_\kk\git_\vartheta G \\
     @AAA @AAA @|\\
   \mathbb{V}(I_{Q})\git_\vartheta G @>>> \mathcal{M}_\vartheta(Q,J_\mathscr{L})@>>> \vert\mathscr{L}\vert
   \end{CD}
 \end{equation}
in which each morphism is a closed immersion. 

Recall that for ideals $B, I\subseteq \kk[y_a : a\in Q_1]$, the saturation of $I$ by $B$ is the ideal 
 \[
 (I : B^\infty) =  \big(g\in \kk[y_a : a\in Q_1] \mid h^Ng \in I \text{ for some }h\in B_Q \text{ and }N>0\big).
 \]

\begin{theorem}
\label{thm:surjective}
If $\mathscr{L}$ is a collection of distinct, basepoint-free line bundles on $X$, then the induced morphism 
\begin{equation}
\label{eqn:morphismtoModuli}
\varphi_{\vert\mathscr{L}\vert}\colon X\longrightarrow \mathcal{M}_\vartheta(Q,J_\mathscr{L})
\end{equation}
is surjective if $I_Q$ coincides with the saturation $(I_{\mathscr{L}} : B_Q^\infty)$. In particular, if $\mathscr{L}$ is very ample and $I_Q = (I_{\mathscr{L}} : B_Q^\infty)$ then \eqref{eqn:morphismtoModuli} is an isomorphism.
\end{theorem}
\begin{proof}
It suffices by Theorem~\ref{thm:closedimmersion} to show that the closed immersion $\mathbb{V}(I_{Q})\git_\vartheta G\to \mathbb{V}(I_{\mathscr{L}})\git_\vartheta G$ is an isomorphism. Proposition~\ref{prop:finemoduli} shows that the ideal $B_Q$ cuts out the $\vartheta$-unstable locus in $\mathbb{A}^{Q_1}_\kk$, so we need only show that $\mathbb{V}(I_{Q})\setminus \mathbb{V}(B_Q)$ is isomorphic to $\mathbb{V}(I_{\mathscr{L}})\setminus \mathbb{V}(B_Q)$. Since $I_Q$ is prime, this holds if  $I_Q = (I_{\mathscr{L}} : B_Q^\infty)$. The second statement is immediate.
\end{proof}

\begin{remark}
\label{rem:moduletheoretic}
In light of Proposition~\ref{prop:algebra} and Remark~\ref{rem:modules}, when the map  \eqref{eqn:morphismtoModuli} is an isomorphism then we describe the Mori Dream Space $X$ as the fine moduli space $\mathcal{M}_\vartheta(\modA)$ of $\vartheta$-stable modules over $A:=\End(\bigoplus_{i\in Q_0} L_i)$ that are isomorphic as $\bigl(
\bigoplus_{i \in Q_0} \kk e_i \bigr)$-modules to $\bigoplus_{i \in Q_0} \kk e_i$.
\end{remark}

\subsection{Digression on multigraded regularity}
The notion of multigraded regularity for a sheaf with respect to a collection of basepoint-free line bundles on a projective variety, introduced by Maclagan--Smith~\cite{MaclaganSmith} in the toric case and Hering--Schenck--Smith~\cite{HSS} in general, provides a multigraded variant of Castelnuovo-Mumford regularity.  The special case of $\mathscr{O}_X$-multigraded regularity provides the key technical tool in the proof of Theorem~\ref{thm:2} (see Section~\ref{sec:theorem2} to follow), and for convenience we recall the definition and an important property.

Consider a list $\mathscr{E}=(E_1,\dots, E_\ell)$ of basepoint-free line bundles on a projective variety $X$. For an integer vector $\beta = (\beta_1,\dots, \beta_\ell)\in \ZZ^\ell$ we set $E^\beta = E_1^{\beta_1}\otimes \cdots \otimes E_\ell^{\beta_\ell}$. 

\begin{definition}
\label{def:regularity}
A line bundle $L$ on $X$ is \emph{$\mathscr{O}_X$-regular with respect to $\mathscr{E}$} if $H^i(X,L\otimes E^{-\beta})=0$ for all $i>0$ and for all $\beta\in \NN^\ell$ with $\beta_1+\cdots+\beta_\ell=i$.
\end{definition}

This notion was used implicitly (see \cite[Proposition 4.14]{CrawSmith}) in the proof of Corollary~\ref{coro:veryample}, but it becomes essential in the following section. The crucial property for us is the following sufficient condition for the multiplication map on global sections to be surjective (see \cite[Theorem~2.1(2)]{HSS} \cite[Theorem~6.9]{MaclaganSmith}):

\begin{proposition}[\cite{HSS, MaclaganSmith}]
\label{prop:regularity}
Let $F$ be a coherent sheaf on $X$. If $F$ is $\mathscr{O}_X$-regular then the map
\[
H^0(X,F\otimes E^\alpha)\otimes_\kk H^0(X,E^\beta)\longrightarrow H^0(X, F\otimes E^{\alpha+\beta})
\]
is surjective for all $\alpha, \beta\in \NN^\ell$.
\end{proposition}

\subsection{Main algebraic result}
\label{sec:theorem2}
We now work towards our main algebraic result which exhibits many collections of line bundles on $X$ for which the morphism from \eqref{eqn:morphismtoModuli} is an isomorphism, thereby providing a noncommutative algebraic construction of $X$ as in Remark~\ref{rem:moduletheoretic}.
	
We first introduce the collections of interest. Choose generators $g_1,\dots, g_m\in\kk[x_1,\dots, x_d]$ of the ideal $I_X$,  set $\delta_0:= \max_{1\leq j\leq m} \big\{ \text{total degree of } g_j\big\}$ and define 
\begin{equation}
\label{eqn:delta}
\delta:=\left\{\begin{array}{cl} \delta_0/2 & \text{if }\delta_0\text{ is even;} \\ (\delta_0+1)/2 & \text{otherwise.}\end{array}\right.
\end{equation}
Consider line bundles $L_1, \dots, L_{r-2}$ on $X$ for which the corresponding rank one reflexive sheaves $E_1=\psi^{-1}(L_1),\dots, E_{r-2}=\psi^{-1}(L_{r-2})$ on $\widetilde{X}$ are basepoint-free line bundles such that the subsemigroup of $\Pic(\widetilde{X})$ generated by $E_1,\dots, E_{r-2}$ contains an ample line bundle. Choose sufficiently positive integers $\beta_1,\dots, \beta_{r-2}$ to ensure that $E:=E_1^{\beta_1}\otimes\cdots\otimes E_{r-2}^{\beta_{r-2}}$ is both normally generated and $\mathscr{O}_{\widetilde{X}}$-regular with respect to $E_1,\dots, E_{r-2}$ (see Definition~\ref{def:regularity}) and, moreover,  that $E^{2\delta}$ is very ample.  Define $E_{r-1}:= E^\delta$ and $E_r:=E^{2\delta}$ (assuming that each is distinct from $E_1,\dots, E_{r-2}$). Augment the list $L_1,\dots, L_{r-2}$ on $X$ with $L_0=\mathscr{O}_X$, $L_{r-1}:=\psi(E_{r-1})$ and $L_r:=\psi(E_r)$ to obtain a collection 
\begin{equation}
\label{eqn:finelist}
\mathscr{L}=(\mathscr{O}_{X}, L_1,\dots, L_r)
 \end{equation}
 of basepoint-free line bundles on $X$. Let $Q$ denote the quiver of sections of $\mathscr{L}$.  The corresponding collection of line bundles $\widetilde{\mathscr{L}}:= (\mathscr{O}_{\widetilde{X}}, E_1, \dots, E_r)$ on $\widetilde{X}$ takes the form of the collections constructed in the proof of \cite[Theorem~5.5]{CrawSmith}, and hence by the proof of that result we have 
\begin{equation}
\label{eqn:toricsaturation}
I_{\widetilde{Q}} = (I_{\widetilde{\mathscr{L}}} : B_Q^\infty)
\end{equation}
and the induced morphism  $\varphi_{\vert\widetilde{\mathscr{L}}\vert}\colon \widetilde{X}\to \mathbb{A}^{Q_1}\git_\vartheta G$ is a closed immersion with image $\mathbb{V}(I_{\widetilde{Q}})\git_\vartheta G$ isomorphic to $\mathcal{M}_\vartheta(Q,J_{\widetilde{\mathscr{L}}})$.

\begin{remark}
\label{rem:finelist}
\begin{enumerate}
\item It follows that each collection \eqref{eqn:finelist} determines a commutative diagram
  \begin{equation}
   \label{eqn:allGITquotients}
   \begin{CD}
\widetilde{X} @>{\varphi_{\vert\widetilde{\mathscr{L}}\vert}}>>  \mathbb{V}(I_{\widetilde{Q}})\git_\vartheta G @>{\cong}>>  \mathcal{M}_\vartheta(Q,J_{\widetilde{\mathscr{L}}})@>>>  \mathbb{A}^{Q_1}_\kk\git_\vartheta G \\
  @AAA   @AAA @AAA @|\\
  X @>{\varphi_{\vert \mathscr{L}\vert}}>> \mathbb{V}(I_{Q})\git_\vartheta G @>>> \mathcal{M}_\vartheta(Q,J_\mathscr{L})@>>> \vert\mathscr{L}\vert
   \end{CD}
 \end{equation}
 in which every morphism is a closed immersion and the indicated map is an isomorphism.
\item Since $E$ is $\mathscr{O}_{\widetilde{X}}$-regular with respect to $E_1,\dots, E_{r-2}$, since each $\beta_i>0$ and since $E_{r-1}=E_1^{\delta\beta_1}\otimes\cdots\otimes E_{r-2}^{\delta\beta_{r-2}}$, Proposition~\ref{prop:regularity} implies that the multiplication map 
\[
H^0(E_{r-1}\otimes E_i^{-1})\otimes_\kk H^0(E_{r-1})\longrightarrow  H^0(E_r\otimes E_i^{-1})
\]
is surjective for $1\leq i\leq r-1$. Since each $E_i$ is invertible, it follows from \eqref{eqn:doubledualHom} that every path in $Q$ from 0 to $r$ passes through $r-1$.
\item For clarity in what follows, we work with elements of $\kk[y_a : a\in Q_1]$ modulo the relation $\sim$ in which polynomials are equivalent when their difference lies in $I_{\widetilde{Q}}$. Since $I_{\widetilde{Q}}$ is the toric ideal of the semigroup homomorphism $\inc\oplus\div \colon \mathbb{N}^{Q_1}\to \Wt(Q)\oplus \NN^d$, monomials satisfy $y^{\textbf{m}}\sim y^{\textbf{m}^\prime}$ if and only if $\inc(\textbf{m}-\textbf{m}^\prime)=0$ and $\div(\textbf{m}-\textbf{m}^\prime)=0$, that is, $y^{\textbf{m}}\sim y^{\textbf{m}^\prime}$ if and only if they  share the same weight in $\Wt(Q)$ and the same image under $\widetilde{\Phi}$. 
\end{enumerate}
\end{remark} 

Before introducing the main result, we present a technical lemma for any list $\mathscr{L}$ as in \eqref{eqn:finelist}. Write $\chi=\sum_i \chi_i\mathbf{e}_i\in \ZZ^{Q_0}$ as $\chi=\chi^+-\chi^-$ where $\chi^{\pm}=\sum_i \chi^{\pm}_i\mathbf{e}_i\in \NN^{Q_0}$ have disjoint supports $I_\chi^+=\{i\in Q_0 : \chi_i>0\}$ and $I_\chi^-=\{i\in Q_0 : \chi_i<0\}$. In particular, $\chi\in \Wt(Q)$ gives 
\[
n_\chi := \sum_{i\in I_\chi^+} \chi^+_i = \sum_{i\in I_\chi^-} \chi^-_i\in \NN.
\]
For any spanning tree $\mathcal{T}$ in $Q$, set $y_{\mathcal{T}}:= \prod_{a\in \supp(\mathcal{T})} y_a$. Recall that any path $p$ in the quiver $Q$ defines the monomial $y_{p}:= \prod_{a\in \supp(p)} y_a$.

 \begin{lemma}
 \label{lem:hardlemma}
Let $\mathcal{T}$ be a spanning tree in $Q$ and let $\chi\in \inc(\mathbb{N}^{Q_1})\setminus\{0\}$. There exists $\mathbf{m}\in \mathbb{N}^{Q_1}$ such that for any monomial $y^{\mathbf{v}}\in \kk[y_a : a\in Q_1]$ of weight $\chi$, we have 
\begin{equation}
\label{eqn:hardlemma}
(y_\mathcal{T})^{2n_\chi}y^{\mathbf{v}} \sim y^{\mathbf{m}} \prod_{\alpha=1}^{n_\chi} y_{\gamma_\alpha}
\end{equation}
where $\gamma_1,\dots, \gamma_{n_\chi}$ are paths in $Q$, each with tail at $0$ and head at $r$.  Also, $y^{\mathbf{v}}$ divides $\prod_{\alpha=1}^{n_\chi} y_{\gamma_\alpha}$, and the resulting quotient $\widetilde{\Phi}(\prod_\alpha y_{\gamma_\alpha})/\widetilde{\Phi}(y^{\mathbf{v}})$ depends only on $\mathcal{T}$ and $\chi$.
\end{lemma}
 \begin{proof}
 We begin by constructing $\mathbf{m}\in \mathbb{N}^{Q_1}$. The spanning tree $\mathcal{T}$ supports a path $q_i$ from $0\in Q_0$ to each vertex $i\in Q_0$ and hence to each vertex in $I_\chi^-$. We may therefore write
 \begin{equation}
\label{eqn:monomialdecomp} 
(y_\mathcal{T})^{n_\chi} = y^{\mathbf{m}_1} \prod_{i\in I_\chi^-} (y_{q_i})^{\chi_i^-}
\end{equation}
where $\mathbf{m}_1 \in \mathbb{N}^{Q_1}$ depends only on $\mathcal{T}$ and $\chi$. The tree $\mathcal{T}$ supports a path $\gamma$ from 0 to $r$ whose label is a torus-invariant section $s\in H^0(E_r)$. Since $E_{r-1}$ is $\mathscr{O}_X$-regular with respect to $E_1,\dots, E_{r-2}$ and each $\beta_i>0$, Proposition~\ref{prop:regularity} implies that the multiplication map
\begin{equation}
\label{eqn:lifts}
H^0(E_{r-1}\otimes E_1^{\delta\beta_1}\otimes \cdots \otimes E_{r-2}^{\delta\beta_{r-2}}\otimes E_j^{-1}) \otimes_\kk H^0(E_j) \to H^0(E_r)
\end{equation}
is surjective. In particular, for each $j\leq r-2$ there exist sections of $E_{r}\otimes E_{j}^{-1}$ and $E_j$ whose product is $s$. Since $Q$ is a complete quiver of sections, there exists a pair of paths labelled by these sections, one from 0 to $j$ denoted $q_j^{\prime\prime}$, and the other from $j$ to $r$ denoted $q_j^{\prime}$. Concatenating gives a path $q_j^{\prime}q_j^{\prime\prime}$ from 0 to $r$ that passes via $j$ and, by Remark~\ref{rem:finelist}(2) through $r-1$, such that $y_\gamma \sim y_{q_j^{\prime}q_j^{\prime\prime}} = y_{q_j^{\prime}}y_{q_j^{\prime\prime}}$. Multiply by $y_\mathcal{T}/y_\gamma\in \kk[y_a : a\in Q_1]$ to obtain $y_{\mathcal{T}}\sim y_{q^{\prime}_j}y^{\mathbf{m}(j)}$ for some $\mathbf{m}(j)\in \NN^{Q_1}$ that depends only on $\mathcal{T}$ and $j$ (and on the lift of $s$ via \eqref{eqn:lifts}, but we fix one such lift for $\mathcal{T}$ and $i$). Applying this $\chi_j^+$-times for each $j\in I_\chi^+$ and multiplying gives 
\[
(y_{\mathcal{T}})^{n_\chi}\sim y^{\mathbf{m}_2} \prod_{j\in I_\chi^+} (y_{q^{\prime}_j})^{\chi_j^+}
\]
for some $\mathbf{m}_2 \in \mathbb{N}^{Q_1}$ depends only on $\mathcal{T}$ and $\chi$. Multiply by \eqref{eqn:monomialdecomp} to see that  
\begin{equation}
\label{eqn:sim}
(y_{\mathcal{T}})^{2n_\chi} \sim y^{\mathbf{m}} \prod_{i\in I_\chi^-} (y_{q_i})^{\chi_i^-} \prod_{j\in I_\chi^+} (y_{q^{\prime}_j})^{\chi_j^+}
\end{equation}
where $\mathbf{m}:=\mathbf{m}_1+\mathbf{m}_2\in \mathbb{N}^{Q_1}$ depends only on $\mathcal{T}$ and $\chi$.

To complete the proof, write $\mathbf{v}=\sum_{a\in Q_1} v_a \mathbf{e}_a\in \mathbb{N}^{Q_1}$ where $\inc(\mathbf{v})=\chi$. Since $\chi\neq 0$ there exists $i\in I_\chi^-$, so there exists $a_1\in Q_1$ with $\tail(a_1)=i$ such that $v_{a_1}>0$. There are two cases. If $\chi_{\head(a_1)}<0$ then $\head(a_1)\in I_\chi^+$, in which case we define $p_{1}:=a_1$ and repeat the above for $\mathbf{v}^\prime:= \mathbf{v}-\mathbf{e}_a$. Otherwise, $\chi_{\head(a_1)}\leq 0$ in which case there exists $a_2\in Q_1$ with $\tail(a_2)=\head(a_1)$ such that $v_{a_2}>0$. Since $Q$ is acyclic we can continue in this way, obtaining a path $p_{1}$ that traverses the arrows $a_1, a_2,\dots $ and satisfies $\chi_{\head(p_{1})}>0$, that is, $\head(p_1)\in I_\chi^+$. As in the first case, we may repeat the above for $\mathbf{v}^\prime:= \mathbf{v}-\sum_{a\in \supp(p_1)} \mathbf{e}_a$. In either case, the weight $\chi^\prime :=\inc(\mathbf{v}^\prime)$ satisfies $n_{\chi^\prime} = n_\chi - 1$, and we obtain by induction a set of paths $p_1,\dots, p_{n_\chi}$ satisfying $y^{\mathbf{v}} = \prod_{\alpha=1}^{n_\chi} y_{p_\alpha}$, where precisely $\chi_i^-$ of these paths have tail at $i\in I_\chi^-$ and $\chi_i^+$ have head at $i\in I_\chi^+$. Thus, for $1\leq \alpha\leq n_\chi$, there exists $i\in I_\chi^-, j\in I_\chi^+$ such that $\gamma_\alpha:=q^{\prime}_jp_\alpha q_i$ is a path in $Q$ from 0 to $r$ and
\[
\prod_{\alpha=1}^{n_\chi} y_{\gamma_\alpha} = \prod_{i\in I_\chi^-} (y_{q_i})^{\chi_i^-} \prod_{\alpha=1}^{n_\chi} y_{p_\alpha} \prod_{i\in I_\chi^+} (y_{q^{\prime}_i})^{\chi_i^+}.
\]
Note that $y^{\mathbf{v}}$ divides $\prod_{\alpha=1}^{n_\chi} y_{\gamma_\alpha}$, and multiplying \eqref{eqn:sim} by $y^{\mathbf{v}}$ establishes \eqref{eqn:hardlemma}. The quotient $\widetilde{\Phi}(\prod_\alpha y_{\gamma_\alpha})/\widetilde{\Phi}(y^{\mathbf{v}})$ equals $\widetilde{\Phi}\big((y_{\mathcal{T}})^{2n_\chi}\big)/\widetilde{\Phi}(y^{\mathbf{m}})$, so depends only on $\mathcal{T}$ and $\chi$ as required.
\end{proof}

\begin{remark}
Applying $\widetilde{\Phi}(-)$ to \eqref{eqn:hardlemma} and dividing the resulting equality by $\widetilde{\Phi}(y^{\mathbf{v}})$ shows in addition that the monomial $\widetilde{\Phi}(y^{\mathbf{m}})$ divides $\widetilde{\Phi}\big((y_{\mathcal{T}})^{2n_\chi}\big)$.
\end{remark}

We now come to our our main algebraic result. The proof is rather involved, so we sketch the outline. The goal is to show that the ideals $I_{\mathscr{L}}, I_Q$ defined by some $\mathscr{L}$ satisfy  $I_Q = (I_{\mathscr{L}} : B_Q^\infty)$. This boils down to showing that for each $f\in I_Q$, there exists both $g\in I_{\mathscr{L}}$ and $N\in \mathbb{N}$ such that $(y_\mathcal{T})^N f \sim g$, where $\sim$ is the equivalence relation from Remark~\ref{rem:finelist}(3). However, the generators of $I_{\mathscr{L}}$ from equation \eqref{eqn:IL} are rather special polynomials of the form $\sum_{p\in \Gamma} c_{p}y_p$, where $\Gamma$ is a finite set of paths with the same head and same tail, such that $\widetilde{\Phi}\big(\sum_{p\in \Gamma} c_{p} y_p\big) \in I_{X}$. The goal, then, is to find an appropriate sets of paths $\{\Gamma_i\}$ for which an element $(y_\mathcal{T})^N f$ is equivalent under $\sim$ to a polynomial combination of elements of the form $\sum_{p\in \Gamma_i} c_{p}y_p$ that satisfy $\widetilde{\Phi}\big(\sum_{p\in \Gamma_i} c_p y_{p}\big)\in I_X$. \textsc{Step 1} below introduces sets of paths $\{\Gamma_i\}$ using Lemma~\ref{lem:hardlemma}, but at this stage we know only that $(y_\mathcal{T})^N f$ is equivalent under $\sim$ to a combination of elements $\sum_{p\in \Gamma_i} c_{p}\prod_{\alpha} y_{p_\alpha}$. In \textsc{Step 2} we use multigraded regularity and the relation $\sim$ to refine every such element so that, in  \textsc{Step 3}, we can pull out a sufficiently large common factor to leave generators of the form $\sum_{p\in \Gamma_i} c_{p}y_p$ as required. At each stage, we take care to control the appropriate element of $I_X$. 

 \begin{theorem}
 \label{thm:main}
Let $L_1,\dots, L_{r-2}$ be distinct, basepoint-free line bundles on $X$ for which the corresponding reflexive sheaves $E_1=\psi^{-1}(L_1),\dots, E_{r-2}=\psi^{-1}(L_{r-2})$ on $\widetilde{X}$ are basepoint-free line bundles. If the subsemigroup of $\Pic(\widetilde{X})$ generated by $E_1,\dots, E_{r-2}$ contains an ample bundle, then there exists $L_{r-1}, L_r\in \Pic(X)$ such that  
\begin{equation}
\label{eqn:isomorphismtoModuli}
\varphi_{\vert\mathscr{L}\vert}\colon X\longrightarrow \mathcal{M}_\vartheta(\modAL)
\end{equation}
is an isomorphism for $\mathscr{L} = (\mathscr{O}_X, L_1, \dots, L_r)$ and $A=\End_{\mathscr{O}_X}(\bigoplus_{0\leq i\leq r} L_i)$. 
  \end{theorem}
\begin{proof} 
Define the line bundles $L_{r-1}$ and $L_r$ as described at the start of this section to produce a collection $\mathscr{L}$ of the form \eqref{eqn:finelist}. Remark~\ref{rem:finelist}(1) shows that $\mathscr{L}$ is very ample, so by Theorem~\ref{thm:surjective} it suffices to prove that $I_Q = (I_{\mathscr{L}} : B_Q^\infty)$. One inclusion is straightforward: let $f\in (I_{\mathscr{L}} : B_Q^\infty)$. Since $I_{\mathscr{L}}\subseteq I_Q$ and hence $(I_{\mathscr{L}} : B_Q^\infty)\subseteq (I_{Q} : B_Q^\infty)$, we have that $(y_\mathcal{T})^Nf\in I_Q$ for some spanning tree $\mathcal{T}$ and $N\in \mathbb{N}$. Since $I_Q$ is prime, we have either $f\in I_Q$ as required, or $B_Q\subseteq I_Q$. This latter conclusion would force the contradiction $\mathbb{V}(I_Q)\git_\vartheta G=\emptyset$, so $(I_{\mathscr{L}} : B_Q^\infty)\subseteq I_Q$. 

For the opposite inclusion, let $f\in I_Q$ be a homogeneous generator of weight $\chi\in \inc(\mathbb{N}^{Q_1})\setminus\{0\}$ and let $\mathcal{T}$ be a spanning tree in $Q$. If we can show that $(y_\mathcal{T})^N f \in I_{\widetilde{Q}} + I_\mathscr{L}$ for some $N\in \mathbb{N}$, then by increasing $N$ if necessary and applying the equality $I_{\widetilde{Q}} = (I_{\mathscr{\widetilde{L}}} : B_Q^\infty)$ from \eqref{eqn:toricsaturation}, we deduce that $(y_\mathcal{T})^N f \in I_{\mathscr{\widetilde{L}}} + I_{\mathscr{L}} = I_{\mathscr{L}}$ and hence $f\in (I_{\mathscr{L}} : B_Q^\infty)$ as required.   In light of Remark~\ref{rem:finelist}(3), we need only find $g\in I_{\mathscr{L}}$ such that $(y_\mathcal{T})^N f \sim g$, and we achieve this for $N=2n_\chi$.

We proceed in four steps:

\smallskip

\textsc{Step 1:} \emph{Introduce a set of paths $\{\gamma_{\alpha, \beta}\}$ in $Q$ such that 
\begin{equation}
\label{eqn:step1}
(y_\mathcal{T})^{2n_\chi}f\sim y^{\mathbf{m}}\bigg(\sum_\beta c_\beta  \prod_{\alpha=1}^{n_\chi} y_{\gamma_{\alpha,\beta}}\bigg)
\end{equation}
for some $\mathbf{m}\in \NN^{Q_1}$ and $c_\beta\in \kk$, where in addition we have $\widetilde{\Phi}\big(\sum_\beta c_\beta \prod_{1\leq \alpha\leq n_\chi} y_{\gamma_{\alpha, \beta}}\big)\in I_X$.}

\smallskip

Decompose $f$ as a sum of terms $f=\sum_{\beta} c_\beta y^{\mathbf{v}_\beta}$ for $c_\beta\in \kk$ and $\mathbf{v}_\beta\in \mathbb{N}^{Q_1}$ satisfying $\chi = \inc(\mathbf{v}_\beta)$. Since $\chi\neq 0$ we apply Lemma~\ref{lem:hardlemma} to each monomial $y^{\mathbf{v}_\beta}$ to obtain $(y_\mathcal{T})^{2n_\chi}y^{\mathbf{v}_\beta} \sim y^{\mathbf{m}} \prod_{\alpha=1}^{n_\chi} y_{\gamma_{\alpha,\beta}}$, where $\mathbf{m}$ depends only on  $\mathcal{T}$ and $\chi$ (not on  $\beta$) and where each $\gamma_{\alpha,\beta}$ is a path in $Q$ with tail at $0$ and head at $r$.  This gives \eqref{eqn:step1}. Also, the quotient $x^\mathbf{q}:= \widetilde{\Phi}(\prod_\alpha y_{\gamma_{\alpha,\beta}})/\widetilde{\Phi}(y^{\mathbf{v}_\beta})\in \kk[x_1,\dots,x_d]$ depends only on  $\mathcal{T}$ and $\chi$ (not on  $\beta$). Since $f\in I_Q$, we have $\widetilde{\Phi}(f)\in I_X$ and hence we deduce that $\widetilde{\Phi}\big(\sum_\beta c_\beta \prod_{\alpha=1}^{n_\chi} y_{\gamma_{\alpha, \beta}}\big) = x^{\mathbf{q}}\big(\sum_\beta c_\beta \widetilde{\Phi}(y^{\mathbf{v}_\beta})\big) = x^{\mathbf{q}}\widetilde{\Phi}(f)\in I_X$ as required.
\medskip

\textsc{Step 2:} \emph{Introduce a second set of paths $\{p_{i,j,k,\ell}\}$ in $Q$ such that 
\[
\sum_\beta c_\beta  \prod_{\alpha=1}^{n_\chi} y_{\gamma_{\alpha,\beta}} \sim \sum_{i,j,k} c_{i,j,k}  \prod_{\ell=1}^{n_\chi} y_{p_{i,j,k,\ell}}
\]
for some $c_{i,j,k}\in \kk$, where for each $i, j$ we have $\widetilde{\Phi}\big(\sum_{k} c_{i,j,k}  \prod_{1\leq \ell \leq n_\chi} y_{p_{i,j,k,\ell}}\big)\in I_X$.}

\smallskip

In light of Step 1, expand $\widetilde{\Phi}\big(\sum_\beta c_\beta \prod_{\alpha=1}^{n_\chi} y_{\gamma_{\alpha, \beta}}\big) = \sum_{i,j}h_{i,j}g_i$ in terms of generators of $I_X$, where each $h_{i,j}\in \kk[x_1,\dots,x_d]$ is a nonzero term. Since $\widetilde{\Phi}$ is equivariant and $y_{\gamma_{\alpha, \beta}}$ has weight $\mathbf{e}_r-\mathbf{e}_0\in \Wt(Q)$, we may assume without loss of generality that each term in this expansion has degree $\widetilde{\pic}(n_\chi (\mathbf{e}_r-\mathbf{e}_0))=E_r^{n_\chi}$. Thus, expanding each $g_i:=g_{i,1}+\cdots +g_{i,t_i}$ as a sum of terms for some $t_i\in \NN$ gives $ h_{i,j} g_{i,k} \in H^0(E_r^{n_\chi})$  for all $i,j,k$. Since $E_{r-1}$ is $\mathscr{O}_X$-regular with respect to $E_1,\dots, E_{r-2}$ and $E_r=E_{r-1}^2$, Proposition~\ref{prop:regularity} implies that the multiplication map 
\[
H^0(E_{r}) \otimes_\kk \dots \otimes_\kk H^0(E_{r}) \to H^0(E_r^{n_\chi})
\]
is surjective, so for each $i,j,k$ there exists $c_{i,j,k}\in \kk$ and torus-invariant sections $s_{i,j,k,\ell}\in H^0(E_r)$ for $1\leq \ell\leq n_\chi$ such that $h_{i,j}g_{i,k} = c_{i,j,k}\prod_{\ell = 1}^{n_\chi} s_{i,j,k,\ell}$.  Since $Q$ is a quiver of sections, there exists a path $p_{i,j,k,\ell}$ in $Q$ from 0 to $r$ whose label is the torus-invariant section $s_{i,j,k,\ell}$, that is, $\widetilde{\Phi}(y_{p_{i,j,k,\ell}}) = s_{i,j,k,\ell}$. For fixed $i, j$, we therefore obtain
\begin{equation}
\label{eqn:hijgik}
h_{i,j}g_{i,k} = c_{i,j,k} \widetilde{\Phi}\bigg(\prod_{\ell = 1}^{n_\chi} y_{p_{i,j,k,\ell}}\bigg).
\end{equation}
Summing over $1\leq k\leq t_i$ gives $h_{i,j}g_i = \widetilde{\Phi}\big(\sum_{k} c_{i,j,k}  \prod_{1\leq \ell \leq n_\chi} y_{p_{i,j,k,\ell}}\big)$, and by summing this new expression over all $i,j$ we deduce that 
\begin{equation}
\label{eqn:Phiequality}
\widetilde{\Phi}\bigg(\sum_\beta c_\beta \prod_{\alpha=1}^{n_\chi} y_{\gamma_{\alpha, \beta}}\bigg) = \widetilde{\Phi}\bigg(\sum_{i,j,k} c_{i,j,k} \prod_{\ell=1}^{n_\chi} y_{p_{i,j,k,\ell}}\bigg) 
\end{equation}
lies in $I_X$ by Step 1. The main statement of Step 2 now follows from Remark~\ref{rem:finelist}(3) because these polynomials also share the same weight in $\Wt(Q)$, namely $n_\chi(\mathbf{e}_r-\mathbf{e}_0)$. 

\medskip

\textsc{Step 3:} \emph{Introduce a third set of paths $\{q_{i,j,k}\}$ in $Q$ such that
\begin{equation}
\label{eqn:step3} 
\prod_{\ell=1}^{n_\chi} y_{p_{i,j,k,\ell}}\sim  y^{\mathbf{m}^\prime_{i,j}} y_{q_{i,j,k}}
\end{equation}
for some $\mathbf{m}^\prime_{i,j}\in \NN^{Q_1}$, where for each $i,j$ we have $\widetilde{\Phi}\big(\sum_k c_{i,j,k} y_{q_{i,j,k}}\big)\in I_X$.}

\smallskip

Fix $i, j$ and simplify notation by suppressing dependence on $i, j$. To this end, write $c_k:=c_{i,j,k}$, $y^{\mathbf{v}_k}:=  \prod_{1\leq \ell \leq n_\chi} y_{p_{i,j,k,\ell}}$, $h=h_{i,j}$ and $g_k=g_{i,k}$, in which case \eqref{eqn:hijgik} is simply $hg_k = c_k\widetilde{\Phi}(y^{\mathbf{v}_k})$. The map $\widetilde{\Phi}$ is equivariant and sends monomials to monomials, so $\frac{1}{c_k}hg_k\in H^0(E_r^{n_\chi})$ defines a torus-invariant section. Since $E=E_1^{\beta_1}\otimes \cdots \otimes E_{r-2}^{\beta_{r-2}}$ is $\mathscr{O}_X$-regular with respect to $E_1,\dots, E_{r-2}$ and $E_r=E_{r-1}^2 = E^{2\delta}$, Proposition~\ref{prop:regularity} implies that the multiplication map 
\[
H^0(E) \otimes_\kk \dots \otimes_\kk H^0(E) \to H^0(E_r^{n_\chi})
\]
 is surjective, so $\frac{1}{c_k}hg_k$ is equal to the product of $2\delta n_\chi$ torus-invariant sections of $E$. Since $g_k$ is a term of a generator of $I_X$, its total degree is at most $\delta_0\leq 2\delta$ by \eqref{eqn:delta}, so we may choose $2\delta$ of these sections $s_{k,1},\dots, s_{k,2\delta}\in H^0(E)$ such that $g_k$ divides $\prod_{1\leq \mu \leq 2\delta} s_{k,\mu}\in H^0(E_{r})$. We now apply the above only for $k=1$. Since $Q$ is a quiver of sections, there exists a path $q_{1}$ in $Q$ from 0 to $r$ satisfying $\widetilde{\Phi}(y_{q_{1}}) = \prod_{1\leq \mu \leq 2\delta} s_{1,\mu}$, so the section $hg_1/c_1\widetilde{\Phi}(y_{q_{1}})\in H^0(E_{r}^{n_\chi-1})$ is torus-invariant. Surjectivity of the multiplication map $H^0(E_r) \otimes_\kk \dots \otimes_\kk H^0(E_r) \to H^0(E_r^{n_\chi-1})$ determines $n_\chi-1$ sections of $E_r$ and hence paths $q^\prime_1,\dots, q^\prime_{n_\chi-1}$ in $Q$ from 0 to $r$ labelled by these sections such that $\widetilde{\Phi}(y^{\mathbf{m}^\prime}) =  hg_1/c_1\widetilde{\Phi}(y_{q_{1}})$ for $y^{\mathbf{m}^\prime}:=\prod_{1\leq \nu \leq n_\chi-1} y_{q^\prime_\nu}$. In particular, 
\begin{equation}
\label{eqn:v1q1}
\widetilde{\Phi}(y^{\mathbf{v}_1}) = \frac{hg_1}{c_1} = \widetilde{\Phi}(y^{\mathbf{m}^\prime}y_{q_1}).
\end{equation}
Both monomials $y^{\mathbf{v}_1}$ and $y^{\mathbf{m}^\prime}y_{q_1}$ have weight $n_\chi(\mathbf{e}_r-\mathbf{e}_0)\in \Wt(Q)$, hence $y^{\mathbf{v}_1} \sim y^{\mathbf{m}^\prime}y_{q_1}$. Now reintroduce the indices $i,j$, setting $\mathbf{m}^\prime_{i,j}:= \mathbf{m}^\prime$ and $q_{i,j,1}:=q_1$, to obtain \eqref{eqn:step3} for $k=1$.

For $k>1$, we have $hg_k = c_k\widetilde{\Phi}(y^{\mathbf{v}_k})$. For $1\leq i\leq m$, the generator $g_i$ of $I_X$ is $T$-homogeneous, so $g_k:=g_{i,k}$ and $g_1:=g_{i,1}$ have the same degree in $\ZZ^\rho$. Since $g_1$ divides $\widetilde{\Phi}(y_{q_{1}})\in H^0(E_r)$, it follows that the term $\widetilde{\Phi}(y_{q_{1}})g_k/g_1$ also has degree $E_r$. Divide by its coefficient $c_k/c_1\in \kk$ to obtain a torus-invariant section $\widetilde{\Phi}(y_{q_{1}})c_1g_k/c_k g_1\in H^0(E_r)$ which in turn determines a path $q_k$ in $Q$ with tail at 0 and head at $r$ for which $\widetilde{\Phi}(y_{q_k}) = \widetilde{\Phi}(y_{q_{1}})c_1g_k/c_k g_1$. Then  \eqref{eqn:v1q1} gives
\[
\widetilde{\Phi}(y^{\mathbf{v}_k}) = hg_1\cdot\frac{g_k}{c_kg_1} =  c_1\widetilde{\Phi}(y^{\mathbf{m}^\prime})\widetilde{\Phi}(y_{q_1})\cdot \frac{g_k}{c_kg_1}  = \widetilde{\Phi}(y^{\mathbf{m}^\prime}y_{q_k}).
\]
It follows that the monomials $y^{\mathbf{v}_k}$ and $y^{\mathbf{m}^\prime}y_{q_k}$ have weight $n_\chi(\mathbf{e}_r-\mathbf{e}_0)$, hence $y^{\mathbf{v}_k} \sim y^{\mathbf{m}^\prime}y_{q_k}$, and we reintroduce the indices $i,j$, setting $q_{i,j,k}:=q_k$, to obtain \eqref{eqn:step3} for all $k$. Then
\[
\widetilde{\Phi}(y^{\mathbf{m}^\prime_{i,j}}) \widetilde{\Phi}\bigg(\sum_{k} c_{i,j,k}  y_{q_{i,j,k}}\bigg) =  \widetilde{\Phi}\bigg(\sum_{k} c_{i,j,k}  \prod_{\ell=1}^{n_\chi} y_{p_{i,j,k,\ell}}\bigg)\in I_X
 \]
holds for every $i,j$ by combining \eqref{eqn:step3} and Step 2. The ideal $I_X$ does not contain the monomial $\widetilde{\Phi}(y^{\mathbf{m}^\prime_{i,j}})$ for any $i,j$, otherwise it would contain a variable of $\kk[x_1,\dots,x_d]$ because $I_X$ is prime, and we may assume this is not the case by redefining $\Cox(X)$ from \eqref{eqn:CoxX} using a polynomial ring in $d-1$ variables. Thus,  $\widetilde{\Phi}\big(\sum_{k} c_{i,j,k}  y_{q_{i,j,k}}\big)\in I_X$ for every $i, j$ as required.
 
\medskip

\textsc{Step 4:} \emph{Establish that $(y_\mathcal{T})^{2n_\chi} f \in I_{\widetilde{Q}} + I_\mathscr{L}$ as required by proving that
\begin{equation}
\label{eqn:step4}
(y_\mathcal{T})^{2n_\chi}f\sim y^{\mathbf{m}}\bigg(\sum_{i,j} y^{\mathbf{m}^\prime_{i,j}} \Big(\sum_{k} c_{i,j,k} y_{q_{i,j,k}}\Big)\bigg).
\end{equation}
}

\smallskip

Relation \eqref{eqn:step4} is immediate from Steps 1-3. For every $i, j$ we also have $\sum_{k} c_{i,j,k} y_{q_{i,j,k}}\in I_{\mathscr{L}}$ by Step 3, so the right hand side of \eqref{eqn:step4} also lies in $I_\mathscr{L}$. The definition of $\sim$ given in Remark~\ref{rem:finelist}(3) then implies $(y_\mathcal{T})^{2n_\chi} f \in I_{\widetilde{Q}} + I_\mathscr{L}$. This completes the proof of Theorem~\ref{thm:main}.
\end{proof}

\begin{example}
\label{ex:Grassmannian}
Consider the Mori Dream Space $X=\Gr(2,4)$, where
\[
\Cox(X) = \frac{\kk[x_1,\dots, x_6]}{(x_1x_6 - x_2x_5 + x_3x_4)}
\]
and $\Pic(X)\cong \ZZ$. Let $L$ denote the very ample line bundle whose sections determine the Pl\"{u}cker embedding in the ambient toric variety $\widetilde{X}=\mathbb{P}^5$. Consider Theorem~\ref{thm:main} for $L_1:= L$. The generator of $I_X$ has total degree 2, so by \eqref{eqn:delta} we have $\delta = 1$.  To construct the collection $\mathscr{L}$ from \eqref{eqn:finelist}, note that $E_1=\mathscr{O}_{\mathbb{P}^5}(1)$ is normally generated and $\mathscr{O}_{\mathbb{P}^5}$-regular with respect to $E_1$. In this case, $E_{r-1}:=E_1^\delta = E_1$ and we need only add $E_r:=E^{2\delta}=\mathscr{O}_{\mathbb{P}^5}(2)$ to the collection to obtain $\widetilde{\mathscr{L}}=(\mathscr{O}_{\mathbb{P}^5}, \mathscr{O}_{\mathbb{P}^5}(1), \mathscr{O}_{\mathbb{P}^5}(2))$ and hence $\mathscr{L} = (\mathscr{O}_X, L, L^2)$. The map $\varphi_{\vert \mathscr{L}\vert} \colon \Gr(2,4)\to \lvert \mathscr{L}\vert$ is a closed immersion with image $\mathbb{V}(I_{Q})\git_\vartheta G$ by Proposition~\ref{prop:morphismbundles}\one\ and Theorem~\ref{thm:closedimmersion}\three.  The quiver of sections $Q$ is shown in Figure~\ref{fig:Grassmannian}, 
  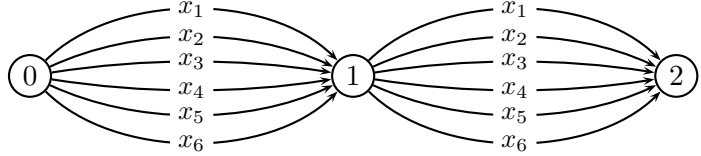
\begin{figure}[!ht]
    \centering
    \mbox{
        \psset{unit=0.85cm}
        \begin{pspicture}(10,2.1) 
            \cnodeput(0,1){A}{0}
            \cnodeput(5,1){C}{1} 
            \cnodeput(10,1){E}{2}
            \psset{nodesep=0pt}
            \nccurve[angleA=45,angleB=135]{->}{A}{C}\lput*{:U}{\small{$x_1$}}
            \nccurve[angleA=25,angleB=155]{->}{A}{C}\lput*{:U}{\small{$x_2$}}
            \nccurve[angleA=9,angleB=171]{->}{A}{C}\lput*{:U}{\small{$x_3$}}
            \nccurve[angleA=351,angleB=189]{->}{A}{C}\lput*{:U}{\small{$x_4$}}
            \nccurve[angleA=335,angleB=205]{->}{A}{C}\lput*{:U}{\small{$x_5$}}  
            \nccurve[angleA=315,angleB=225]{->}{A}{C}\lput*{:U}{\small{$x_6$}}
            \nccurve[angleA=45,angleB=135]{->}{C}{E}\lput*{:U}{\small{$x_1$}}
            \nccurve[angleA=25,angleB=155]{->}{C}{E}\lput*{:U}{\small{$x_2$}}
            \nccurve[angleA=9,angleB=171]{->}{C}{E}\lput*{:U}{\small{$x_3$}}
            \nccurve[angleA=351,angleB=189]{->}{C}{E}\lput*{:U}{\small{$x_4$}}
            \nccurve[angleA=335,angleB=205]{->}{C}{E}\lput*{:U}{\small{$x_5$}}  
            \nccurve[angleA=315,angleB=225]{->}{C}{E}\lput*{:U}{\small{$x_6$}}
         \end{pspicture}
         }
   \caption{The quiver of sections $Q$ for the collection $\mathscr{L}$ \label{fig:Grassmannian}}
  \end{figure}
where we list the arrows with tail at 0 as $y_1,\dots, y_{6}$ where $\div(y_i) = x_i$ for $1\leq i\leq 6$, and those with tail at 1 as $y_7,\dots, y_{12}$ where $\div(y_{i+6}) = x_i$ for $1\leq i\leq 6$. The algorithm from Winn~\cite[Chapter~5]{Winn} gives 
\[
I_{\mathscr{L}}= I_{Q} = \left(\begin{array}{c} 
 y_{1}y_{8}-y_{2}y_{7}, y_{1}y_{9}-y_{3}y_{7}, y_{1}y_{10}-y_{4}y_{7},y_1y_{11}-y_{5}y_{7}\\ y_1y_{12}-y_{6}y_{7}, y_2y_{9}-y_{3}y_{8}, y_2y_{10}-y_4y_{8},y_2y_{11}-y_5y_{8} \\ y_2y_{12}-y_6y_{8}, y_3y_{10}-y_4y_{9}, y_3y_{11}-y_5y_{9}, y_3y_{12}-y_6y_{9} \\
 y_4y_9-y_2y_{11}+y_1y_{12}, y_4y_{11}-y_5y_{10}, y_4y_{12}-y_6y_{10}, y_5y_{12}-y_6y_{11} \\
 y_9y_{10}-y_8y_{11}+y_7y_{12},  y_3y_{10}-y_2y_{11}+y_1y_{12},  y_3y_4-y_2y_5+y_1y_6
 \end{array}\right),
 \]
so the natural inclusion map $\mathbb{V}(I_{Q})\git_\vartheta G \to \mathbb{V}(I_{\mathscr{L}})\git_\vartheta G$ from diagram \eqref{eqn:allGITquotients} is an isomorphism. Thus,  $\varphi_{\vert\mathscr{L}\vert}\colon \Gr(2,4)\rightarrow \mathcal{M}_\vartheta(\modA)$ is an isomorphism for $A=\End_{\mathscr{O}_X}(\bigoplus_{0\leq i\leq 2} L_i)$.
\end{example}

\begin{remark}
We did not have to saturate $I_{\mathscr{L}}$ by $B_Q$ in order to obtain $I_Q$ in Example~\ref{ex:Grassmannian}, but this is necessary for the choice $L_1=L^2$. Indeed, the collection in this case is $\mathscr{L}=(\mathscr{O}_X, L^2, L^4)$, the ideals $I_{\mathscr{L}}$ and $I_Q$ are computed explicitly by Winn~\cite[Section~6.4]{Winn} and are shown to satisfy  $I_{Q}=(I_{\mathscr{L}} : B_Q^\infty)$ even though $I_Q\neq I_{\mathscr{L}}$. The ideals have very many generators, so we do not reproduce them here. Instead we illustrate the saturation aspect of the construction with a pair of examples in the next section. 
\end{remark} 

\section{Reconstructing del Pezzo surfaces from a tilting bundle}
\label{sec:5}
As an application of our main results we illustrate how to reconstruct del Pezzo surfaces directly from the bound quiver of sections of full, strongly exceptional collection of line bundles.

\subsection{Tilting bundles on del Pezzo surfaces}
Let $X$ be a smooth projective variety over $\kk$ and write $\coh(X)$ for the category of coherent sheaves on $X$. For any vector bundle $\mathscr{T}$ on $X$, let $A:=\End_{\mathscr{O}_{X}}(\mathscr{T})$ denote its endomorphism algebra and $\modA$ the abelian category of finitely generated right $A$-modules. We say that $\mathscr{T}$ is a \emph{tilting bundle} on $X$ if the functor
\[
 \Rderived\Hom(\mathscr{T},-)\colon
 D^b\big(\!\coh(X)\big)\longrightarrow D^b\big(\!\modA\big)
 \]
 is an exact equivalence of bounded derived categories. If $\mathscr{T}$ decomposes as a direct sum of line bundles $\mathscr{T}=\bigoplus_{0\leq i\leq r} L_i$ (we need not assume that each $L_i$ has rank one, but we choose to), then after reordering if necessary, the collection $(L_0,L_1,\dots, L_r)$ is a full, strongly exceptional collection on $X$. That is,  the line bundles in the collection generate $D^b(\coh(X))$ as a triangulated category and they satisfy appropriate $\Ext$-vanishing conditions, namely, that $\Hom(L_j,L_i)=0$ for $j>i$ and that $\Ext^k(L_i,L_j) = 0$ for $k>0$ and all $0\leq i,j\leq r$.

For $0\leq k\leq 8$, let $X_k$ denote the del Pezzo surface obtained as the blow-up of $\mathbb{P}^2_\kk$ at $k$ points in general position. The Picard group $\ZZ^{k+1}$ has a basis given by $H$, the pullback to $X_d$ of the hyperplane class on $\mathbb{P}^2_\kk$,  together with the $k$ exceptional curves $R_1, \dots, R_k$. Consider the sequence of basepoint-free line bundles 
\begin{equation}
\label{eqn:fsec}
\mathscr{L}_k :=\big(\mathscr{O}_{X_{k}}, H, 2H-R_1, \ldots, 2H-R_k, 2H\big)
\end{equation}
on $X_k$, and write $L_0=\mathscr{O}_{X_{k}}$, $L_1= H$,  $L_{i+1}=2H-R_i$ for $1\leq i\leq k$, and $L_{k+2}=2H$. 

\begin{lemma}
\label{lem:veryample}
The collection $\mathscr{L}_k$ is very ample for $0\leq k\leq 8$.
\end{lemma}
\begin{proof}
For $k=4$, the collection $\mathscr{L}_4$ is precisely the collection \eqref{eqn:fsecX4}, and $\varphi_{\vert \mathscr{L}_4\vert}$ was shown to be a closed immersion in Example~\ref{exa:X4part2}. The proof in that example is valid for any $0\leq k\leq 8$.
\end{proof}

\begin{remark}
In the terminology of Bergman--Proudfoot~\cite{BergmanProudfoot}, Lemma~\ref{lem:veryample} asserts that $\vartheta$ is \emph{great}.
\end{remark}

The following result is well known to experts, but for convenience we guide the reader towards a proof using the technology of toric systems introduced by Hille-Perling~\cite{HillePerling}. A \emph{toric system} is a (cyclically ordered) list of Cartier divisors whose sum is the anticanonical divisor, such that the intersection product of any two adjacent divisors equals 1 and the intersection product of any distinct and non-adjacent divisors equals 0.

\begin{proposition}
The line bundles  \eqref{eqn:fsec} on $X_k$ form a full,  strongly exceptional collection, that is, the vector bundle $\mathscr{T}_k:= \bigoplus_{0\leq i\leq k+2}L_i$ is tilting.
\end{proposition}
\begin{proof}
Beginning with the unique toric system $H, H, H$ on $\mathbb{P}^2_\kk$, construct a toric system on each $X_k$ as follows:  choose $H, H-R_1, R_1, H-R_1$ on $X_1$, repeat for $k\geq 2$, introducing $R_k$ in the second-last position while subtracting $R_k$ from each neighbouring divisor to obtain the system
\[
H, H-R_1, R_1-R_2, R_2-R_3, \dots, R_{k-1}-R_k, R_k, H-\sum_{1\leq i\leq k} R_i
\]
on $X_k$. List these divisors from left to right as $D_1, \dots, D_{k+3}$. Observe that for $1\leq i\leq k+2$ we have $L_i = \mathscr{O}(D_1+\cdots + D_{i})$, and $-K_{X_k} = \mathscr{O}(D_1+\cdots + D_{k+3})$, and Hille--Perling~\cite[Theorem~5.7]{HillePerling} establishes that $(L_0,L_1,\dots, L_{k+2})$ is full, strongly exceptional sequence as required.
\end{proof}

Thus far then, the variety $X_k$ is smooth, the vector bundle $\mathscr{T}_k$ is tilting, and the morphism $\varphi_{\vert \mathscr{L}_k\vert} \colon X_k\to \vert \mathscr{L}_k\vert$ is a closed immersion. A result of Bergman--Proudfoot~\cite[Theorem~2.4]{BergmanProudfoot} now shows that $\varphi_{\vert \mathscr{L}_k\vert}$ identifies $X_k$ with a connected component of the appropriate moduli space:

\begin{corollary}
\label{cor:bergmanProudfoot}
For $0\leq k\leq 8$ and  for the algebra $A_k:=\kk Q/J_{\mathscr{L}_k}$, the morphism $\varphi_{\vert \mathscr{L_k}\vert}$ identifies $X_k$ with a connected component of the moduli space $\mathcal{M}_\vartheta(\modALk)$. 
\end{corollary}

We now apply our main results to refine the statement of Corollary~\ref{cor:bergmanProudfoot}.  Let $(Q_k, J_k)$ denote the bound quiver of sections of the collection $\mathscr{L}_k$ on $X_k$. For $k\leq 3$, the variety $X_k$ is a toric variety, in which case $\mathscr{L}_k=\widetilde{\mathscr{L}}_k$ and the method of Craw--Smith~\cite{CrawSmith} shows that the morphism $\varphi_{\vert\mathscr{L}_k\vert}\colon X_k\to \mathcal{M}_\vartheta(\modALk)$ is an isomorphism in each case. The Mori Dream Space $X_k$ is not toric for $k>3$ and we now investigate a pair of examples.

 \subsection{The four-point blow-up}
 On the del Pezzo $X_4$, the collection 
 \[
 \mathscr{L}_4 :=\big(\mathscr{O}_{X_{4}}, H, 2H-R_1, 2H-R_2, 2H-R_3, 2H-R_4, 2H\big)
\]
 from \eqref{eqn:fsec} is precisely the collection listed in  \eqref{eqn:fsecX4}, and the quiver of sections $Q$ is shown in Figure~\ref{fig:X4easy}. In this case, the irrelevant ideal in $\mathbb{A}^{Q_1}$ is given by the intersection of monomial ideals
 \[
 B_{Q} = (y_1,\ldots, y_6)\cap (y_7,y_8,y_9)\cap(y_{10},y_{11}, y_{12})\cap(y_{13},y_{14},y_{15})\cap(y_{16},y_{17},y_{18})\cap(y_{19},\dots ,y_{22}),
 \] 
 and the ideal $I_{Q}$ is given in Example~\ref{exa:X4part2}. Winn~\cite[Chapter~5]{Winn} computes 
\[
I_{\mathscr{L}_4}= \left(\begin{array}{c} 
 y_{15}y_{21}-y_{18}y_{22}, y_{12}y_{20}-y_{17}y_{22}, y_{11}y_{20}-y_{14}y_{21},y_9y_{19}-y_{16}y_{22}\\
 y_8y_{19}-y_{13}y_{21}, y_7y_{19}-y_{10}y_{20},y_6y_{17}-y_5y_{18},y_6y_{16}-y_3y_{18}\\
 y_5y_{16}-y_3y_{17},y_6y_{14}-y_4y_{15},y_6y_{13}-y_2y_{15},y_4y_{13}-y_2y_{14},y_5y_{11}-y_4y_{12}\\
 y_5y_{10}-y_1y_{12},y_4y_{10}-y_1y_{11}, y_3y_8-y_2y_9,y_3y_7-y_1y_9,y_2y_7-y_1y_8\\
 y_3y_{14}y_{21}-y_4y_{16}y_{22},y_5y_{13}y_{21}-y_2y_{17}y_{22}, y_6y_{10}y_{20}-y_1y_{18}y_{22}\\
 y_{16}-y_{17}+y_{18}, y_{13}-y_{14}+y_{15},y_{10}-y_{11}+y_{12},y_7-y_8+y_9,y_3-y_5+y_6 \\
 y_2-y_4+y_6, y_1-y_4+y_5,y_{15}y_{21}-y_{18}y_{22},y_{12}y_{20}-y_{17}y_{22},y_{11}y_{20}-y_{14}y_{21}\\
 y_9y_{19}-y_{17}y_{22}+y_{18}y_{22}, y_8y_{19}-y_{14}y_{21}+y_{18}y_{22},y_6y_{17}-y_5y_{18}\\
y_6y_{14}-y_4y_{15}, y_5y_{11}-y_4y_{12},y_5y_8-y_6y_8-y_4y_9+y_6y_9
    \end{array}\right),
 \]
and a Macaulay2~\cite{M2} computation gives $I_{Q}=(I_{\mathscr{L}_4} : B_Q^\infty)$ in this case. Since $\mathscr{L}_4$ is very ample, Theorem~\ref{thm:surjective} implies that the closed immersion $\varphi_{\vert\mathscr{L}_4\vert}\colon X_4\rightarrow \mathcal{M}_\vartheta(\modA_4)$ is an isomorphism.

\subsection{The five-point blow-up}
To obtain $X_5$ we blow-up $\mathbb{P}^2_\kk$ at the torus-invariant points in addition to $(1,1,1)$ and $(1,2,3)$, in which case $\Cox(X_5)$ is the quotient of $\kk[x_1, \dots, x_{16}]$ by
\[ 
I_{X_5} = \left( \begin{array}{c} 
x_5x_{16}+x_6x_{13}-3x_8x_{10}, x_4x_{16}+2x_6x_{14}+x_7x_{12}, x_4x_{16}+x_6x_{14}+x_9x_{10} \\
x_3x_{16}+x_6x_{15}+x_8x_{12}, x_3x_{16}+2x_6x_{15}+x_9x_{11}, x_2x_{16}+x_7x_{15}-2x_8x_{14} \\
x_2x_{16}+3x_7x_{15}+2x_9x_{13}, x_1x_{16}+2x_{10}x_{15}-x_{11}x_{14}, x_1x_{16}+3x_{10}x_{15}+x_{12}x_{13} \\ 
x_2x_6-x_3x_7+x_4x_8, 2x_2x_6-3x_3x_7-x_5x_9, x_1x_6-x_3x_{10}+x_4x_{11} \\ 
x_1x_6-3x_3x_{10}-x_5x_{12}, x_1x_7-x_2x_{10}+x_4x_{13}, x_1x_7-2x_2x_{10}+x_5x_{14}\\ 
x_1x_8-x_2x_{11}+x_3x_{13}, -2x_1x_8+x_2x_{11}-x_5x_{15}, -x_1x_9+2x_2x_{12}+3x_3x_{14}\\ 
-2x_1x_9+x_2x_{12}+3x_4x_{15}, x_6x_{13}-x_7x_{11}+x_8x_{10} 
\end{array}\right),
 \]
see Winn~\cite[Chapter~2]{Winn}. On $X_5$, the full strongly exceptional collection \eqref{eqn:fsec} is 
 \[
 \mathscr{L}_5 :=\big(\mathscr{O}_{X_{5}}, H, 2H-R_1, 2H-R_2, 2H-R_3, 2H-R_4, 2H-R_5, 2H\big)
\]
and the quiver of sections $Q$ is shown in Figure~\ref{fig:X5} (in fact we omit one arrow labelled $x_1x_2x_4x_5x_{16}$ with tail at 0 and head at $4$ to prevent the figure from becoming illegible). Arrows with tail at 0 and head at 1 are listed $a_1,\dots, a_{10}$ from the top of Figure~\ref{fig:X5} to the bottom;  
     \begin{figure}[!ht]
    \centering
      \psset{unit=1.3cm}
      \begin{pspicture}(-0.6,-0.3)(10,4.5)
      \cnodeput(-0.5,2){AA}{0}
      \cnodeput(3.5,2){A}{1}
      \cnodeput(8,4){B}{2} 
      \cnodeput(8,3){C}{3}
       \cnodeput(8,2){D}{4}
       \cnodeput(8,1){E}{5} 
      \cnodeput(8,0){G}{6}
     \cnodeput(10,2){F}{7}
      \psset{nodesep=0pt}
       \nccurve[angleA=66,angleB=114]{->}{AA}{A}\lput*{:U}{\tiny{$x_1x_2x_{6}$}}
       \nccurve[angleA=48,angleB=132]{->}{AA}{A}\lput*{:U}{\tiny{$x_1x_3x_7$}}
        \nccurve[angleA=32,angleB=148]{->}{AA}{A}\lput*{:U}{\tiny{$x_1x_4x_8$}}
        \nccurve[angleA=18,angleB=162]{->}{AA}{A}\lput*{:U}{\tiny{$x_1x_5x_9$}}
        \nccurve[angleA=6,angleB=174]{->}{AA}{A}\lput*{:U}{\tiny{$x_2x_3x_{10}$}}
         \nccurve[angleA=354,angleB=186]{->}{AA}{A}\lput*{:U}{\tiny{$x_2x_4x_{11}$}}
        \nccurve[angleA=342,angleB=198]{->}{AA}{A}\lput*{:U}{\tiny{$x_2x_5x_{12}$}}
        \nccurve[angleA=328,angleB=212]{->}{AA}{A}\lput*{:U}{\tiny{$x_3x_4x_{13}$}}
        \nccurve[angleA=312,angleB=228]{->}{AA}{A}\lput*{:U}{\tiny{$x_3x_5x_{14}$}}
        \nccurve[angleA=294,angleB=246]{->}{AA}{A}\lput*{:U}{\tiny{$x_4x_5x_{15}$}}
        \nccurve[angleA=85,angleB=160]{->}{A}{B}\lput*{:U}(0.5){\tiny{$x_2x_6$}}
        \nccurve[angleA=77,angleB=170]{->}{A}{B}\lput*{:U}(0.65){\tiny{$x_3x_7$}}
        \nccurve[angleA=68,angleB=182]{->}{A}{B}\lput*{:U}(0.55){\tiny{$x_4x_8$}}
        \nccurve[angleA=60,angleB=190]{->}{A}{B}\lput*{:U}(0.65){\tiny{$x_5x_9$}}
        \nccurve[angleA=50,angleB=160]{->}{A}{C}\lput*{:U}(0.35){\tiny{$x_1x_6$}}
        \nccurve[angleA=40,angleB=170]{->}{A}{C}\lput*{:U}(0.65){\tiny{$x_3x_{10}$}}
        \nccurve[angleA=30,angleB=180]{->}{A}{C}\lput*{:U}(0.35){\tiny{$x_4x_{11}$}}
        \nccurve[angleA=25,angleB=190]{->}{A}{C}\lput*{:U}(0.65){\tiny{$x_5x_{12}$}}
        \nccurve[angleA=13,angleB=167]{->}{A}{D}\lput*{:U}(0.35){\tiny{$x_1x_{7}$}}
        \nccurve[angleA=5,angleB=175]{->}{A}{D}\lput*{:U}(0.65){\tiny{$x_2x_{10}$}}
        \nccurve[angleA=355,angleB=185]{->}{A}{D}\lput*{:U}(0.35){\tiny{$x_4x_{13}$}}
        \nccurve[angleA=347,angleB=193]{->}{A}{D}\lput*{:U}(0.65){\tiny{$x_5x_{14}$}}
        \nccurve[angleA=335,angleB=170]{->}{A}{E}\lput*{:U}(0.65){\tiny{$x_1x_{8}$}}
        \nccurve[angleA=330,angleB=180]{->}{A}{E}\lput*{:U}(0.35){\tiny{$x_2x_{11}$}}
        \nccurve[angleA=320,angleB=190]{->}{A}{E}\lput*{:U}(0.65){\tiny{$x_3x_{13}$}}
        \nccurve[angleA=310,angleB=200]{->}{A}{E}\lput*{:U}(0.35){\tiny{$x_5x_{15}$}}
        \nccurve[angleA=300,angleB=170]{->}{A}{G}\lput*{:U}(0.65){\tiny{$x_1x_9$}}
        \nccurve[angleA=292,angleB=178]{->}{A}{G}\lput*{:U}(0.55){\tiny{$x_2x_{12}$}}
        \nccurve[angleA=283,angleB=187]{->}{A}{G}\lput*{:U}(0.7){\tiny{$x_3x_{14}$}}
        \nccurve[angleA=273,angleB=197]{->}{A}{G}\lput*{:U}(0.5){\tiny{$x_4x_{15}$}}
        \ncline{->}{B}{F}\lput*{:U}{\tiny{$x_1$}}    
        \ncline{->}{C}{F}\lput*{:U}{\tiny{$x_2$}}    
        \ncline{->}{D}{F}\lput*{:U}{\tiny{$x_3$}}    
        \ncline{->}{E}{F}\lput*{:U}{\tiny{$x_4$}}    
        \ncline{->}{G}{F}\lput*{:U}{\tiny{$x_5$}}    
        \nccurve[angleA=70,angleB=170]{->}{AA}{B}\lput*{:U}(0.5){\tiny{$x_2x_3x_4x_5x_{16}$}}    
        \nccurve[angleA=65,angleB=170]{->}{AA}{C}\lput*{:U}(0.5){\tiny{$x_1x_3x_4x_5x_{16}$}}    
        \nccurve[angleA=295,angleB=190]{->}{AA}{E}\lput*{:U}(0.5){\tiny{$x_1x_2x_3x_5x_{16}$}}    
        \nccurve[angleA=290,angleB=190]{->}{AA}{G}\lput*{:U}(0.5){\tiny{$x_1x_2x_3x_4x_{16}$}}    
         \end{pspicture}    \caption{Almost the quiver of sections for $\mathscr{L}_5$ (one arrow omitted from 0 to 4)}
  \label{fig:X5}
  \end{figure}
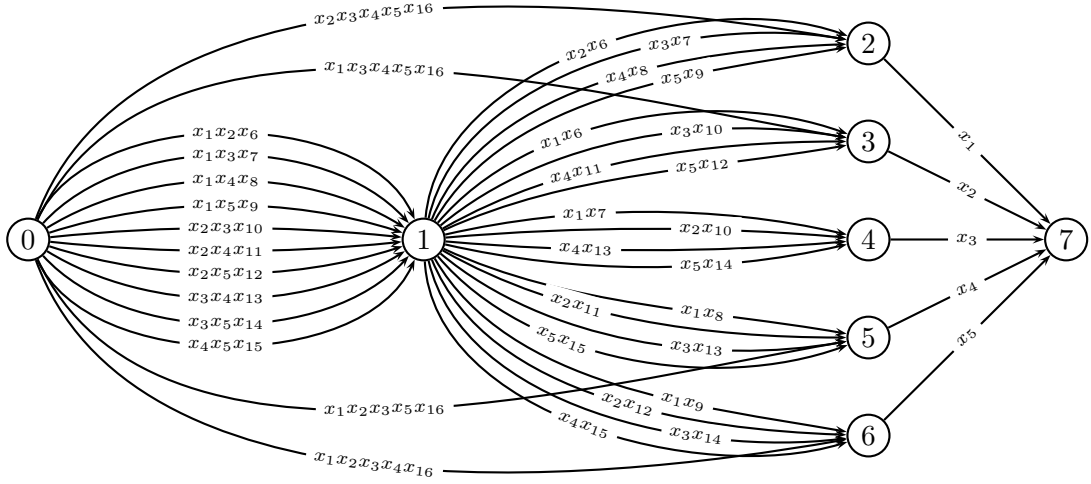 
  list those with tail at 1 as $a_{11}, \dots, a_{30}$ from top to bottom;  list those with head at 7 as $a_{31}, \dots, a_{35}$ from top to  bottom; and list those with tail at 0 and head at $i\geq 2$ as $a_{36}, \dots, a_{40}$ from top to bottom, where the arrow omitted from the figure is $a_{38}$. List the coordinates of $\mathbb{A}^{Q_1}_\kk$ as $y_1,\dots, y_{40}$, and compute the kernel of \eqref{eqn:Psibar} to obtain the ideal
\[
I_Q= \left(\begin{array}{c} y_7+2y_9-y_{10},\; y_6-2y_8+y_{10},\; y_5+y_8-y_9,\; y_4+y_9-2y_{10},\; y_3-y_8+y_1\\
y_2+2y_8-y_9,y_1+3y_8-y_9-y_{10},\; y_{28}+2y_{29}-y_{30},\; y_{27}+y_{29}-2y_{30} \\
y_{24}-2y_{25}+y_{26},\; y_{23}-y_{25}+y_{26},\; y_{20}+y_{21}-y_{22}\\
y_{19}+2y_{21}-y_{22},\; 2y_{16}+y_{17}+y_{18},\; 2y_{15}+3y_{17}+y_{18}\\
y_{12}+2y_{13}+y_{14} ,y_{11}+3y_{13}+y_{14},\; {10}y_{29}-y_9y_{30},\; y_{14}y_{31}+y_9-2y_{10} \\
2y_{8}y_{29}+2y_{8}y_{30}-3y_{9}y_{30}-y_{40},\; y_{10}y_{25}-y_8y_{26}, \; y_{13}y_{31}-y_8+y_{10} \\
2y_9y_{25}+2y_8y_{26}-3y_{9}y_{26}-y_{39},\; 2y_{10}y_{21}+2y_8y_{22}-3y_{10}y_{22}-y_{38}\\
y_9y_{21}-y_8y_{22},\; 3y_{10}y_{17}-2y_8y_{18}+3y_{10}y_{18}-2y_{37}, \; y_{21}y_{26}y_{29}-y_{22}y_{25}y_{30} \\
3y_9y_{17}+2y_8y_{18}-y_{37},\; 6y_{10}y_{13}-2y_8y_{14}+3y_{10}y_{14}-y_{36}\\
3y_9y_{13}+y_8y_{14}-y_{36},\; y_{30}y_{35}-y_{10},\; y_{29}y_{35}-y_9 ,y_{26}y_{34}-y_{10} \\
y_{25}y_{34}-y_8 ,y_{22}y_{33}-y_9,\; y_{21}y_{33}-y_8,\; y_{18}y_{32}+2y_9-y_{10},\; y_{17}y_{32}-2y_8+y_{10}\\
  \end{array}\right).
\]
Details of the computation are given by Winn~\cite[Chapter~6]{Winn}, where it is also shown that 
\[
I_{\mathscr{L}_5}= \left(\begin{array}{c} y_7+2y_9-y_{10} ,y_6-2y_8+y_{10} , y_5+y_8-y_9 , y_4+y_9-2y_{10}, y_3-y_8+y_{10}   \\
y_2+2y_8-y_9 , y_1+3y_8-y_9-y_{10}, y_{28}+2y_{29}-y_{30} , y_{27}+y_{29}-2y_{30} \\
y_{24}-2y_{25}+y_{26}  ,y_{23}-y_{25}+y_{26}   ,y_{20}+y_{21}-y_{22}   ,y_{19}+2y_{21}-y_{22} \\
2y_{16}+y_{17}+y_{18},2y_{15}+3y_{17}+y_{18} ,y_{12}+2y_{13}+y_{14},y_{11}+3y_{13}+y_{14}\\
y_{10}y_{29}-y_9y_{30},2y_8y_{29}+2y_8y_{30}-3y_9y_{30}-y_{40} ,y_{10}y_{25}-y_8y_{26}  \\
2y_9y_{25}+2y_8y_{26}-3y_9y_{26}-y_{39} ,2y_{10}y_{21}+2y_8y_{22}-3y_{10}y_{22}-y_{38} \\
y_9y_{21}-y_8y_{22}  ,3y_{10}y_{17}-2y_8y_{18}+3y_{10}y_{18}-2y_{37} ,3y_9y_{17}+2y_8y_{18}-y_{37} \\ 6y_{10}y_{13}-2y_8y_{14}+3y_{10}y_{14}-y_{36}  ,3y_9y_{13}+y_8y_{14}-y_{36} ,y_{26}y_{34}-y_{30}y_{35} \\
y_{22}y_{33}-y_{29}y_{35} ,y_{21}y_{33}-y_{25}y_{34} , y_{18}y_{32}+2y_{29}y_{35}-y_{30}y_{35}   \\
y_{17}y_{32}-2y_{25}y_{34}+y_{30}y_{35}    , y_{14}y_{31}+y_{29}y_{35}-2y_{30}y_{35}     \\
y_{13}y_{31}-y_{25}y_{34}+y_{30}y_{35}
  \end{array}\right).
\]
A computation using Macaulay2 gives that $I_{Q}=(I_{\mathscr{L}_5} : B_Q^\infty)$ in this case. Since $\mathscr{L}_5$ is very ample, Theorem~\ref{thm:surjective} implies that the morphism $\varphi_{\vert\mathscr{L}_5\vert}\colon X_5\rightarrow \mathcal{M}_\vartheta(\modA_5)$ is an isomorphism.

\end{document}